\documentclass[11pt]{amsart}

\usepackage{amsmath,amssymb,latexsym,soul,cite,mathrsfs}

\usepackage{color,enumitem,graphicx}
\usepackage[colorlinks=true,urlcolor=blue,
citecolor=red,linkcolor=blue,linktocpage,pdfpagelabels,
bookmarksnumbered,bookmarksopen]{hyperref}
\usepackage[english]{babel}

\usepackage[left=2.9cm,right=2.9cm,top=2.8cm,bottom=2.8cm]{geometry}
\usepackage[hyperpageref]{backref}

\usepackage[colorinlistoftodos]{todonotes}
\makeatletter
\providecommand\@dotsep{5}
\def\listtodoname{List of Todos}
\def\listoftodos{\@starttoc{tdo}\listtodoname}
\makeatother

\numberwithin{equation}{section}
\newtheorem{theorem}{Theorem}[section]
\newtheorem{proposition}[theorem]{Proposition}
\newtheorem{definition}[theorem]{Definition}
\newtheorem{lemma}[theorem]{Lemma}

\newtheorem{claim}[theorem]{Claim}

\newcommand\R{\mathbb R}
\newcommand\N{\mathbb N}

\newcommand\K{\mathbb K}

\begin{document}

\title[Existence of solution for a class of variational ...]
{Existence of solution for a class of variational inequality in whole $\R^N$ with critical growth: \linebreak The local Mountain pass case}

\author{Claudianor O. Alves}
\author{Luciano M. Barros}
\author{C\'esar E. Torres Ledesma}
%\author{Gaetano Siciliano}

\address[Claudianor O. Alves]{\newline\indent Unidade Acad\^emica de Matem\'atica
\newline\indent 
Universidade Federal de Campina Grande,
\newline\indent
58429-970, Campina Grande - PB - Brazil}
\email{\href{mailto:coalves@dme.ufcg.edu.br}{coalves@dme.ufcg.edu.br}}

\address[Luciano M. Barros]{\newline\indent Unidade Acad\^emica de F\'isica e Matem\'atica
\newline\indent 
Universidade Federal de Campina Grande,
\newline\indent
58175-000, Cuit\'e - PB - Brazil}
\email{\href{lucianomb@mat.ufcg.edu.br}{lucianomb@mat.ufcg.edu.br}}

\address[C\'esar E. Torres Ledesma]
{\newline\indent Departamento de Matem\'aticas
\newline\indent 
Universidad Nacional de Trujillo
\newline\indent
Av. Juan Pablo II s/n. Trujillo-Per\'u}
\email{\href{ctl\_576@yahoo.es}{ctl\_576@yahoo.es}}

\pretolerance10000

%\begin{document}

\begin{abstract}
In this paper we study the existence of solution for a class of variational inequality in whole $\mathbb{R}^N$ where the nonlinearity has a critical growth for $N \geq 2$. By combining  a penalization scheme found in del Pino and Felmer \cite{Delpino} with a  penalization method due to Bensoussan and Lions \cite{Bensoussan}, we improve a recent result by  Alves, Barros and Torres \cite{CALBCT}.  

\end{abstract}

%\thanks{Claudianor Alves was partially supported by CNPq/Brazil Proc. 304036/2013-7 ; Giovany M. Figueiredo was partially
%supported by  CNPq, Brazil; Gaetano Siciliano  was partially supported by
%Fapesp and CNPq, Brazil. }

\maketitle

%------------------------------------------------------------------------------
\section{Introduction}
%------------------------------------------------------------------------------

In this paper we deal with the existence of nontrivial weak solutions for the following variational inequality   
\begin{equation}\label{01}
\begin{cases}
u\in \mathbb{K}&\\
\displaystyle\int_{\mathbb{R}^N}\nabla u \nabla(v-u)dx+\int_{\mathbb{R}^N}(1+\lambda V(x))u(v-u)dx 
\geq \int_{\mathbb{R}^N}f(u)(v-u)dx,&\forall v \in \mathbb{K} 
\end{cases}
\end{equation}
where $\lambda > 0$ is a positive parameter, 
$$
\mathbb{K} = \lbrace v \in E; \ v \geq \varphi \ \text{a.e. in}\ \Omega \rbrace,
$$ 
with $\varphi \in H^1(\R^N)$, $\varphi^+ \not=0$, $supp \, \varphi^+ \subset \Omega$,
$$
E=\Big\{u \in H^{1}(\mathbb{R}^{N}); \ \int_{\mathbb{R}^{N}} V(x)|u|^{2}dx<+ \infty\Big\},
$$  
and $V\in C(\mathbb{R}^{N}, \mathbb{R})$ satisfies the following conditions:  
\begin{itemize}
	\item[$(V_{1})$] $V(x) \geq 0, \,\, \forall x \in \mathbb{R}^N$;
	
	\item[$(V_{2})$] $\Omega := int (V^{-1}(\{0\})) \neq \emptyset$ is an open, connected and bounded set 
	of $\mathbb{R}^{N}$ with smooth boundary;
\end{itemize}
With respect to the function $f:\mathbb{R} \rightarrow \mathbb{R}$, let us assume the following conditions:
If $N\geq 3$, $f$ is of the form 
$$
f(t)=\mu |t|^{q-2}t+|t|^{2^*-2}t, \quad \forall t \in \mathbb{R}, \leqno{(f_0)}
$$
where $2^*=\frac{2N}{N-2}$ and $2<q<2^*$.

\noindent 
If $N=2$,  let us  assume that $f\in C^1(\R,\R)$ and that it has an exponential growth, that is, there is $\alpha_0>0$ such that 
$$ \displaystyle\lim_{|t|\rightarrow\infty}\displaystyle\frac{|f(t)|}{e^{\alpha t^2}}=
\left\{
\begin{array}{cc}
0,& \alpha>\alpha_0,\\
+ \infty,& \alpha<\alpha_0.\\
\end{array}
\right.
( \mbox{see Figueiredo, Miyagaki and Ruf\cite{Figueiredo}})$$
Furthermore, we also suppose  the conditions below:
\begin{itemize}
	\item[$(f_{1})$] $\displaystyle\frac{f(t)}{t} \rightarrow 0$  as  $|t| \rightarrow 0 $;
	
	\item[$(f_{2})$] There is $\theta > 2$  such that 
	$$0< \theta F(t):=\theta\displaystyle\int_{0}^{t}f(\sigma)d\sigma
	\leq f(t)t, \ t\neq0;$$
	\item[$(f_{3})$] There is $C>0$ such that 
	$$|f'(t)|\leq C e^{\alpha_0 t^2}, \ t\in\mathbb{R};$$
	\item [$(f_4)$] There are $p>1$ and $\nu>0$ such that 
	$$f(t)\geq \nu t^p, \ \forall t\geq0.$$
	
\end{itemize}

\vspace{0.5 cm}
Variational inequality appears in a lot of areas of mathematical and engineering sciences
including elasticity, transportation and economics equilibrium, nonlinear programming,
operations research, see \cite{Bensoussan}, \cite{Crank}, \cite{DL},  \cite{Fichera}, \cite{LS} and \cite{Rodrigues}, which has motivated its study in the last years. The main tools used to study variational inequalities are variational methods, fixed point theory and sub-super solution. In the books due to  Carl, Le and Motreanu \cite{Motreanu} and Motreanu and R\u{a}dulescu \cite{MR} the reader can find an overview involving these methods 

Related to the variational methods, we essentially have  three methods that were used in the study of variational inequalities, which are mentioned below: \\

\noindent {\bf Method 1.} Some authors have used the framework of the nonsmooth critical point theory for lower semicontinuous functionals found in Corvellec,  Degiovanni and Marzocchi \cite{CDM}, Degiovanni and Marzocchi \cite{DM}, Campa and Degiovanni \cite{CD} and Degiovanni and Zani \cite{DZ}. In this direction, we would like to cite the paper due  Magrone, Mugnai and Servadei \cite{MMS}. \\ 

\noindent {\bf Method 2.} Other authors have opted by using the minimax principles for lower semicontinuous functions developed by Szulkin \cite{Szulkin}, see for example,  Alves and  C\^orrea \cite{A&C}, Figueiredo, Furtado and Montenegro \cite{FFM},  Jianfu \cite{Jianfu,Jianfu2}, Mancini and Musina \cite{MM}. \\  

\noindent {\bf Method 3.}  The last method that we would like to cite is the penalization method due to Bensoussan and Lions \cite{Bensoussan}, which was improved by  Carl,  Le and Motreanu \cite{Motreanu}. This method consists in considering a penalized problem where the classical variational methods can be used to solve it. After some estimates it is possible to get a solution for the original problem. This method was used for example in Alves, Barros and Torres \cite{CALBCT}, Magrone and Servadei \cite{Servadei}, Matzeu and Servadei \cite{Servadei3,Servadei2,Servadei4}. \\

Finally, still related to variational inequalities, we would like to cite Baiocchi and Capelo \cite{BC}, Chipot and Yeressian \cite{Chipot}, Fang \cite{Fang}, Friedman \cite{Friedman}, Kinderlehrer and Stampacchia \cite{Stampacchia}, Monneau \cite{Regis}, Panagiotopoulos \cite{Pana},  Teka \cite{Teka}, Troianiello \cite{Troianiello} and their references. \\

Now we are in position to state our main result.

\begin{theorem}\label{main1}
Suppose that $V$ satisfies $(V_1)-(V_2)$ and $f$ satisfies $(f_1)-(f_4)$ if $N=2$. Then there exist $\lambda_*>0$, $\mu_*>0$, $\nu_*>0$ and $\rho = \rho_\mu>0$ or $\rho = \rho_\nu>0$ with $\rho_\mu \to 0$ and $\rho_\nu \to 0$ as $\mu \to +\infty$ and $\nu \to +\infty$ respectively such that, if $||\varphi||<\rho$, problem (\ref{01}) has at least one nontrivial weak solution $u_\lambda$ for $\lambda\geq\lambda_*$, $\mu\geq\mu_*$ and $\nu\geq\nu_*$. Furthermore, for any sequence $\lambda_n\to+\infty$, there exists a subsequence still denote by $(\lambda_{n})$ such that $(u_{\lambda_n})$ converge strongly in $H^1(\R^N)$ to a function $u$ with $u\equiv 0$ in $\mathbb{R}^N \setminus \Omega$, which $u$ is a solution of the following variational inequality	
	$$
	\displaystyle\int_{\Omega}\nabla u \nabla (v-u)\hspace{0.05cm}dx+\displaystyle\int_{\Omega}u(v-u)\hspace{0.05cm}dx\geq
	\int_{\Omega}f(u)u(v-u)\hspace{0.05cm}dx,
	$$
	for every $v \in \tilde{\mathbb{K}}$, where
	$$
	\tilde{\mathbb{K}} = \lbrace v \in H^1_0(\Omega); \ v \geq \varphi \ \text{a.e. in }\ \Omega \rbrace.
	$$
\end{theorem}

In order to prove Theorem \ref{main1}, we have simultaneously used two penalty methods: The first penalization is due to Bensoussan and Lions \cite{Bensoussan} and improved by Carl,  Le and Motreanu \cite{Motreanu}, and the second one is due to del Pino and Felmer \cite{Delpino}. We would like point out that Theorem \ref{main1} improves the main result found in  \cite{CALBCT}, because in that paper the authors have assumed that there is $M_0>0$ such that, the set $\mathcal{L}=\{x\in\mathbb{R}^N; \ V(x) \leq M_0\}$ is nonempty with finite measure, which permits to get a compactness control when $\lambda$ is  large enough.  Here, we overcome the loss of this condition, by using the penalization method found in del Pino and Felmer \cite{Delpino}. Moreover, the Theorem \ref{main1} also complements the study made in Alves and Corr\^ea \cite{A&C}, because in that paper the authors considered $N=1$, the conditions $(V_1)-(V_2)$, and  used nonsmooth critical point theory to prove the existence of a nontrivial solution when $\lambda$ is large enough, while in the present paper, $N \geq 2$, $f$ has a critical growth and we are working with  the penalization method due to Bensoussan and Lions \cite{Bensoussan}.

The plan of the paper is as follows: In Section 2 we show some preliminary results. In Section 3 it is proved the Theorem \ref{main1} for the case $N \geq 3$, while in Section 4 we show  the case $N=2$.

Before concluding this introduction, we would like point out that, without loss of generality, we are supposing that $f(t)=0$, for $t\leq0$.

%%%%%%%%%%%%%%%%%%%%%%%%%%%%%%%%%%%%%%%%%%%%%%%%%

\section{Preliminary results}

In whole this paper, $E_{\lambda}$ denotes the function space defined as
$$
E_\lambda=\Big\{u \in H^{1}(\mathbb{R}^{N}); \ \int_{\mathbb{R}^{N}} V(x)|u|^{2}dx<+ \infty\Big\}
$$
endowed with the inner product  
$$
\langle u,v\rangle_\lambda=\displaystyle\int_{\mathbb{R}^{N}}[\nabla u \nabla v + (1+\lambda V(x))uv]dx,
$$
whose the associated norm is given by  
$$
\|u\|_\lambda=\Bigg(\displaystyle\int_{\mathbb{R}^{N}}[|\nabla u|^2 + (1+\lambda V(x))|u|^2]dx\Bigg)^{\frac{1}{2}}.
$$
Then, for any $\lambda\geq 0$, $(E_\lambda, \langle \;,\;\rangle_\lambda)$ is a Hilbert space and by $(V_1)-(V_2)$,   
$$
\|u\| \leq \|u\|_{\lambda}, \quad \forall u \in E_\lambda. 
$$
From this, the embedding $E_\lambda \hookrightarrow H^{1}(\mathbb{R}^N)$ is continuous for every $\lambda >0$. As a consequence, the embedding   
\begin{equation}\label{embe}
E_\lambda \hookrightarrow L^{s}(\mathbb{R}^N), \quad \forall s \in [2,2^*],
\end{equation}
is also continuous for all $\lambda>0$.

In order to solve problem (\ref{01}), we employ the penalization method due to Carl, Le and Motreanu \cite{Motreanu}. To begin with, let us recall the definition of penalty operator.
\begin{definition}\label{penalty}
	Let $C\neq \emptyset$ a closed and convex subset of a reflexive Banach space $X$. The operator $P:X\rightarrow X'$ is called a penalty operator associated with  $C\subset X$ if: 
	\begin{itemize}
		\item [$(P_1)$] The function 
		$$
		\begin{aligned}
		{}[0,1] & \mapsto \R\\
		t&\mapsto\langle P(u+tv),w\rangle
		\end{aligned}
		$$
		is continuous for every $u$, $v$, $w\in X$;
		
		\item [$(P_2)$] For every $u$, $v\in X$;
		$$
		\langle P(u)-P(v), u-v\rangle\geq0, 
		$$
		\item [$(P_3)$]
		$P(u)=0$ if and only if $u\in C$;
		\item [$(P_4)$] Let $A\subset X$ be a bounded set, then $P(A)$ is bounded.
	\end{itemize}
	\end{definition}
	
Associated to problem (\ref{01}), we have the penalized problem defined as  	
\begin{equation}\label{P01}
-\Delta u+(1+\lambda V(x))u-\frac{1}{\epsilon}(\varphi-u)^+\chi_\Omega=f(u), \quad \text{in} \ \ \R^N.
\end{equation}
We say that $u$ is a weak solution of (\ref{P01}) if:
\begin{equation}\label{P02}
\displaystyle\int_{\mathbb{R}^N}[\nabla u \nabla v+(1+\lambda V(x))uv]\hspace{0.05cm}dx -\displaystyle\frac{1}{\epsilon} \int_\Omega(\varphi-u)^{+}v\hspace{0.05cm}dx=
\int_{\mathbb{R}^N}f(u)v\hspace{0.05cm}dx, \\\ \forall v \in E_\lambda,
\end{equation}	
where $\epsilon>0$  is the penalization parameter and $\displaystyle\frac{1}{\epsilon}\int_\Omega(\varphi-u)^{+}vdx$ is the penalization term.

Hereafter, we will consider the operator $P:E_\lambda\rightarrow E'_\lambda$ defined as
$$
\langle P(u),v\rangle=-\displaystyle\int_\Omega(\varphi-u)^{+}v\hspace{0.05cm}dx, 
$$
which is a penalty operator associated with the convex set $\mathbb{K}$.

Next, for the reader's convenience, we state a result due to Lions \cite{Lions2} that will apply an important rule in our approach when $N \geq 3$.
\begin{lemma}\label{C-C}
Let $(u_n) \subset H^1(\R^N)$ be a bounded sequence such that 
$$
\begin{aligned}
&u_n \rightharpoonup u\;\;\mbox{in}\;\;L^{2^*}(\R^N);\\
&|u_n|^{2^*} \rightharpoonup \nu\;\;\mbox{in}\;\;\mathcal{M}(\R^N)\;\;\mbox{and}\\
&|\nabla u_n|^2 \rightharpoonup \xi \;\;\mbox{in}\;\;\mathcal{M}(\R^N),
\end{aligned}
$$
where $\nu$ and $\xi$ are non negative and finite measure in $\R^N$. Then, there are sequences $(x_n) \subset \R^N$ and $(\nu_n ) \subset [0,\infty)$ such that 
$$
|u_n|^{2^*} \rightharpoonup |u|^{2^*} + \sum_{i=1}^{\infty} \nu_i \delta_{x_i} = \nu,
$$
with 
\begin{equation}\label{C-C1}
\sum_{i=1}^{\infty} \nu_{n}^{\frac{2}{2^*}} < \infty
\end{equation}
and 
\begin{equation}\label{C-C2}
\xi (x_n) \geq S \nu_n^{2/2^*}\;\;\forall n\in \N,
\end{equation}
where $\delta_i$ is the Dirac measure and $S$ is the best Sobolev constant.
\end{lemma}

%%%%%%%%%%%%%%%%%%%%%%%%%%%%%%%%%%%%%%%%%%%%%%%%%

\section{Proof of the Theorem \ref{main1}: Case $N\geq3$}

In this section, we deal with the following variational inequality
\begin{equation}\label{eq3.1}
\begin{cases}
u\in\mathbb{K}&\\
\displaystyle\int_{\mathbb{R}^N}\nabla u \nabla(v-u)\hspace{0.05cm}dx+\displaystyle\int_{\mathbb{R}^N}(1+\lambda V(x))u(v-u)\hspace{0.05cm}dx 
\geq \int_{\mathbb{R}^N}(\mu u^{q-2}+u^{2^*-2})u(v-u)\hspace{0.05cm}dx, \ \forall v \in \mathbb{K}.
\end{cases}
\end{equation}

In order to prove the existence of nontrivial weak solutions for (\ref{eq3.1}),  we adapt some ideas found in \cite{A&C} and \cite{Delpino}. Have this in mind, let us consider the functions $h, H: \R \to \R$ defined by
$$
h(t) = \begin{cases}
f(t),& t\leq a\\
\frac{t}{k},&t\geq a
\end{cases}
\quad \mbox{and}\quad 
H(t) = \int_{0}^{t} h(s)ds,
$$
where $k = \frac{2q}{q-2}>2$, and $a>0$ satisfies $f(a) = \frac{a}{k}$. Now, let $\widetilde{\Omega}$  be a bounded, connected set such that $\overline{\Omega} \subset \widetilde{\Omega}$ and let $g, G: \R^N \times \R \to \R$ defined as
$$
g(x,t) = \chi_{\widetilde{\Omega}}(x)f(t) + (1-\chi_{\widetilde{\Omega}}(x))h(t)
$$ 
and 
$$
G(x,t) = \chi_{\widetilde{\Omega}}(x)F(t) + (1-\chi_{\widetilde{\Omega}}(x))H(t),
$$
where 
$$
F(t) = \int_{0}^{t}f(s)ds.
$$
Then, we deal with the modified variational inequality
\begin{equation}\label{eq3.2}
\begin{cases}
u\in \mathbb{K}&\\
\displaystyle\int_{\mathbb{R}^N}\nabla u \nabla(v-u)dx+\int_{\mathbb{R}^N}(1+\lambda V(x))u(v-u)dx 
\geq \int_{\mathbb{R}^N}g(x,u)(v-u)dx,&\forall v \in \mathbb{K}. 
\end{cases}
\end{equation}

If $u$ is a non-negative solution of (\ref{eq3.2}) such that 
\begin{equation}\label{eq3.2-}
u(x) \leq a\;\;\forall x\in \R^N \setminus \widetilde{\Omega},
\end{equation}
then $u$ is a solution of (\ref{01}).

Next we state some properties of the function $g$, whose the proof follows as in \cite{Delpino}. 
\begin{proposition}\label{prop3.1}
Since $a>0$, the following propositions are true:
\begin{enumerate}
\item[$(g_1)$] $g(x,t) = 0$, for all $x\in \R^N$ and $t\leq 0$.
\item[$(g_2)$] $\frac{g(x,t)}{t} \to 0$ as $t\to 0$, uniformly in $x\in \R^N$.
\item[$(g_3)$] Given $\beta >0$, there exists $C_\beta >0$ such that 
$$
|g(x,t)| \leq \mu \beta |t|^{q-1} + C_\beta |t|^{2^*-1},\;\;\forall (x,t)\in \R^N \times \R.
$$
\item[$(g_4)$] For all $x\in \widetilde{\Omega}$ and $t>0$, we have 
$$
0< qG(x,t) \leq tg(x,t).
$$
\item[$(g_5)$] For all $x\in \R^N \setminus \widetilde{\Omega}$ and $t>0$ we have 
$$
0< 2G(x,t) \leq tg(x,t) \leq \frac{1}{k} (1+\lambda V(x))t^2.
$$
\end{enumerate}
\end{proposition}

Now, we are ready to fix the penalized problem associated to (\ref{eq3.2}) given by  
\begin{equation}\label{eq3.5}
-\Delta u + (1+\lambda V(x))u -\frac{1}{\epsilon}(\varphi-u)^+\chi_\Omega= g(x,u),\quad \mbox{in} \quad \mathbb{R}^N.
\end{equation}
Associated to problem (\ref{eq3.5}), we have the functional $I_{\lambda,\epsilon}: E_\lambda\rightarrow\mathbb{R}$ defined as 
$$
I_{\lambda,\epsilon}(u)=\frac{1}{2}||u||^2_\lambda+\frac{1}{2\epsilon}\displaystyle\int_\Omega[(\varphi-u)^{+}]^2\hspace{0.05cm}dx-\int_{\mathbb{R}^N}G(x,u)\,dx.
$$
It is standard to show that $I_{\lambda, \epsilon} \in C^1(E_\lambda, \R)$ with 
$$
I_{\lambda,\epsilon}'(u)v=\langle u,v\rangle_\lambda-\frac{1}{\epsilon} \int_\Omega(\varphi-u)^+v\,dx- \int_{\mathbb{R}^N}g(x,u)vdx,\;\;\forall u,v \in E_\lambda.
$$
From now on, we say that $u\in E_\lambda$ is a weak solution of (\ref{eq3.5}) if $u$ is a critical point of $I_{\lambda, \epsilon}$.

%-----------------------------------------------------------------------------------------

\subsection{Solution to the penalized problem (\ref{eq3.5})}

We star our analysis by showing that, under the hypothesis of Theorem \ref{main1}, $I_{\lambda, \epsilon}$ satisfies the geometry conditions of the mountain pass. 
\begin{lemma}\label{mpgcon} 
\begin{itemize}
\item[(a)]There exist constant $r_\mu,\rho_\mu >0$, with $r_\mu, \rho_\mu \to 0$ as $\mu \to +\infty$, independent of $\lambda$ and $\epsilon$, such that
	$$
	I_{\lambda,\epsilon}(u) \geq \rho_\mu \quad \mbox{for} \quad \|u\|_{\lambda}=r_\mu;
	$$
\item[(b)] There is $e \in H^1(\mathbb{R}^N)$ with $\|e\|_{\lambda} >r_\mu$ and $I_{\lambda, \epsilon}(e)<0$.
\end{itemize}
\end{lemma}
\begin{proof} 
\begin{itemize}
\item[(a)] By Sobolev embedding and $(g_3)$, fixed $\beta >0$, there are $C_\beta >0$ and positive constants $C_1, C_2$ such that 
$$
\begin{aligned}
I_{\lambda, \epsilon}(u) &\geq \frac{1}{2}\|u\|_{\lambda}^{2} - \frac{\mu \beta}{q}\int_{\R^N} |u|^qdx - \frac{C_\beta}{2^*}\int_{\R^N} |u|^{2^*}dx\\
%&\geq \frac{1}{2}\|u\|_\lambda^2 - \frac{C_1\mu}{2}\|u\|_{\lambda}^{2} - C_2\|u\|_{\lambda}^{2^*}\\
&\geq \frac{1}{2}\|u\|_\lambda^2 - \frac{C_1\mu}{q}\|u\|_{\lambda}^{q} - \frac{C_2}{2^*}\|u\|_{\lambda}^{2^*}
\end{aligned}
$$
Since $2<q<2^*$, fixing $r_\mu>0$  satisfying  
	$$
	r_\mu<\min\left\{\left(\frac{2^*}{4C_2}\right)^{\frac{1}{2^*-2}},\left(\frac{q}{4C_1\mu}\right)^{\frac{1}{q-2}}\right\},
	$$
	we obtain 	
	$$
	I_{\lambda,\epsilon}(u)\geq \frac{1}{8}r_\mu^2:=\rho_\mu \quad \text{for}\quad \|u\|_\lambda=r_\mu.
	$$
This proves $(a)$.

\item[(b)] Since $supp\varphi^+ \subset \Omega$, it is easy to see that 
$$
\int_{\Omega} [(\varphi - \varphi^+)^+]^2dx = 0\;\;\mbox{and}\;\;\|\varphi^+\|_{\lambda} = \|\varphi^+\|,
$$
and so, 
$$
I_{\lambda, \epsilon}(\varphi^+) \leq \|\varphi^+\|^2 \leq \|\varphi\|^2.
$$
Now, fixing $\|\varphi\|$ small enough with respect to $\rho_\mu$, we get 
$$
I_{\lambda, \epsilon}(\varphi^+) < \rho_\mu.
$$
Moreover, using again that $supp(\varphi^+) \subset \Omega$, 
$$
\int_{\R^N} G(x,t\varphi^+)dx = \int_{\R^N} F(t\varphi^+)dx = \frac{\mu t^q}{q}\int_{\tilde{\Omega}}|\varphi^+|^qdx - \frac{t^{2^*}}{2^*} \int_{\tilde{\Omega}} |\varphi^+|^*dx, \quad \mbox{for} \quad t>1.
$$
Then
$$
I_{\lambda, \epsilon}(t\varphi^+) = \frac{t^2}{2}\|\varphi^+\|_{\lambda}^{2} - \frac{t^q \mu}{q}\int_{\tilde{\Omega}} |\varphi^+|^qdx - \frac{t^*}{2^*}\int_{\tilde{\Omega}} |\varphi^+|^{2^*}dx,
$$
which implies that $I_{\lambda, \epsilon}(t\varphi^+) \to -\infty$ as $t\to +\infty$, because $q\in (2,2^*)$. Therefore, taking $t_*$ large enough and $w = (1+t_*)\varphi^+$, 
$$
I_{\lambda, \epsilon}(w) < I_{\lambda, \epsilon}(\varphi^+)< \rho_\mu \leq I_{\lambda, \epsilon} (u),\;\;\mbox{with}\;\;\|u\|_{\lambda} = r_\mu
$$
and
$$
\|\varphi^+\|_{\lambda} < r_\mu < \|w\|_{\lambda}.
$$
\end{itemize}
\end{proof}

\noindent 
By Lemma \ref{mpgcon} and the Mountain Pass Theorem found in Willem \cite{Willem}, there is $(u_n) \subset E_\lambda$ such that 
$$
I_{\lambda, \epsilon}(u_n) \to c_{\lambda, \epsilon}\quad\mbox{and}\quad I'_{\lambda, \epsilon}(u_n) \to 0\;\;\mbox{as}\;\;n\to +\infty,
$$
where
$$
c_{\lambda, \epsilon} = \inf_{\gamma \in \Gamma}\max_{t\in [0,1]} I_{\lambda, \epsilon}(\gamma (t))
$$
and
$$
\Gamma = \{\gamma \in C([0,1], E):\;\gamma(0) = \varphi^+\;\;\mbox{and}\;\;\gamma (1) = w\}.
$$
Arguing as in \cite[Lemma 3.2]{CALBCT}, we have the lemma below 
\begin{lemma}\label{bound1}
There are $\tau >0$ and $\mu_* = \mu_*(\tau)>0$ such that 
$$
c_{\lambda, \epsilon} \in \left( 0, \frac{q-2}{2q}S^{N/2} - \tau\right),\quad \mbox{for all}\;\; \lambda >0, \epsilon >0\;\;\mbox{and}\;\;\mu \geq \mu_*.
$$
\end{lemma}

The next lemma establishes the boundedness of the $(PS)$ sequences of $I_{\lambda, \epsilon}$. 
\begin{lemma}\label{boundps}
If $(w_n)$ is a $(PS)_{c_{\lambda, \epsilon}}$ sequence for the functional $I_{\lambda, \epsilon}$, then $(w_n)$ is bounded in $E_\lambda$.
\end{lemma}
\begin{proof}
First of all, note that 
$$
\begin{aligned}
I_{\lambda, \epsilon}(w_n) - \frac{1}{q}I'_{\lambda, \epsilon}(w_n)w_n & = \left(\frac{1}{2} - \frac{1}{q} \right)\|w_n\|_{\lambda}^{2} + \frac{1}{2\epsilon} \int_{\Omega} [(\varphi - w_n)^+]^2dx\\
&+\frac{1}{\epsilon q} \int_{\Omega} (\varphi - w_n)^+w_ndx + \frac{1}{q}\int_{\R^N} (g(x,w_n)w_n - qG(x,w_n))dx.
\end{aligned}
$$
As $q>2$ and $[(\varphi - w_n)^+]^2 + (\varphi - w_n)^+w_n \geq (\varphi - w_n)^+\varphi,$ it follows that 
$$
\frac{1}{2\epsilon}\int_{\Omega}[(\varphi- w_n)^+]^2dx + \frac{1}{q\epsilon} \int_{\Omega}(\varphi-w_n)^+w_ndx \geq \frac{1}{q\epsilon}\int_{\Omega}(\varphi - w_n)^+\varphi dx.
$$
Furthermore, 
$$
\int_{\R^N}g(x,w_n)w_ndx \geq \int_{\widetilde{\Omega}} g(x,w_n)w_ndx,
$$
and by $(g_4)$ and $(g_5)$,  
$$
\frac{1}{q}\int_{\tilde{\Omega}} (g(x,w_n)w_n - qG(x, w_n))dx \geq 0
$$
and 
$$
\int_{\R^N\setminus \tilde{\Omega}} G(x,w_n)dx \leq \frac{1}{2k}\int_{\R^N\setminus \tilde{\Omega}}(1+\lambda V(x))w_n^2dx \leq \frac{1}{2k} \|w_n\|_{\lambda}^{2}.
$$
Consequently, by H\"older inequality and Sobolev embedding,   
\begin{equation}\label{eq3.6}
\begin{aligned}
I_{\lambda, \epsilon} (w_n) - \frac{1}{q}I'_{\lambda, \epsilon}(w_n)w_n %&\geq \left(\frac{1}{2}-\frac{1}{q} - \frac{1}{2k} \right)\|w_n\|_{\lambda}^{2} + \frac{1}{\epsilon q} \int_{\Omega}(\varphi - w_n)^+\varphi dx\\
&\geq \frac{1}{2}\left( \frac{1}{2} - \frac{1}{q} \right)\|w_n\|_{\lambda}^{2} - \frac{1}{\epsilon q}|\varphi|_{2}^{2} - \frac{C}{\epsilon q}|\varphi|_{2}\|w_n\|_{\lambda},
\end{aligned}
\end{equation}
for some positive constant $C.$ By other hand, there is $n_0\in \N$ such that 
\begin{equation}\label{eq3.7}
I_{\lambda, \epsilon}(w_n) - \frac{1}{q}I'_{\lambda, \epsilon}(w_n)w_n \leq c_{\lambda, \epsilon} + o_n(1) + o_n(1)\|w_n\|_{\lambda},\;\;n\geq n_0.
\end{equation}
Thus, since $k = \frac{2q}{q-2}$, from (\ref{eq3.6}) and (\ref{eq3.7}), 
$$
\frac{1}{2}\left( \frac{1}{2} - \frac{1}{q}\right)\|w_n\|_{\lambda}^{2} \leq c_{\lambda, \epsilon} + o_n(1) + \frac{1}{\epsilon q}|\varphi|_{2}^{2} + \left( o_n(1) + \frac{C}{\epsilon q}|\varphi|_{2}^{2} \right)\|w_n\|_{\lambda},
$$
showing the desired result. 
\end{proof}

The next two lemmas can be proof as in \cite[Lemmas 3.4 and 3.5]{CALBCT}, and so, their proofs are omitted.    

\begin{lemma}\label{nonneg}
Without loss of generality, we can suppose that $w_n \geq 0$ for all $n\in \N$, that is, $(w_n^+)$ be a $(PS)_{c_{\lambda, \epsilon}}$ sequence for $I_{\lambda, \epsilon }$
\end{lemma}

\begin{lemma}\label{bound2}
Let $(w_n)_{n\in \N}$ be a $(PS)_{c_{\lambda, \epsilon}}$ sequence for the functional $I_{\lambda,\epsilon}$. Then
	$$
	\displaystyle\limsup_{n\to+\infty} ||w_n||^2_\lambda\leq \frac{4q}{q-2}c_{\lambda, \epsilon}.
	$$
\end{lemma}

The lemma below is crucial to prove that the functional $I_{\lambda,\epsilon}$ satisfies the $(PS)$ condition in some levels. 

\begin{lemma}\label{vanish}
Let $(w_n)$ be a $(PS)_{c_{\lambda, \epsilon}}$ sequence for the functional $I_{\lambda,\epsilon}$, then, given $\delta >0$ there exists $R>0$ such that 
$$
\limsup_{n\to \infty} \int_{B^{c}(0,R)}(|\nabla w_n|^2 + |w_n|^2)dx < \delta.
$$
\end{lemma}
\begin{proof}
In what follows, fixed $R>0$ such that $\widetilde{\Omega} \subset B(0, \frac{R}{2})$, let us set the function $\eta \in C^1(\R^N, \R)$ given by  
$$
\begin{aligned}
&\eta (x)\in [0,1],\;\;\forall x\in \R^N;\\
&\eta(x) = 0, \;\forall x\in \overline{B}(0, \frac{R}{2});\\
&\eta(x) = 1,\;\forall x\in B^c(0,R)\;\;\mbox{and}\\
& |\nabla \eta(x)| \leq \frac{M_1}{R},\;\forall x\in \R^N,
\end{aligned}
$$
It is immediate  to show that $(\eta w_n) \subset  E_\lambda$ and that it is a bounded sequence in $E_\lambda$. Taking $\eta w_n$ as a test function, we find  
%$$
%I'_{\lambda, \epsilon}(u_n)(\eta u_n) = o_n(1),
%$$
%namely 
\begin{equation}\label{eq--}
\int_{\R^N} [\nabla w_n \nabla (\eta w_n) + (1+\lambda V(x))\eta w_n^2]dx - \frac{1}{\epsilon}\int_{\Omega} (\varphi - w_n)^+\eta w_ndx - \int_{\R^N} g(x,w_n)(\eta w_n)dx = o_n(1).
\end{equation}
Since 
$$
\int_{\Omega}(\varphi - w_n)^+(\eta w_n)dx = 0,
$$
the equality (\ref{eq--}) can be written as
$$
\int_{\R^N} [\nabla w_n \nabla (\eta w_n) + (1+\lambda V(x))\eta w_n^2]dx - \int_{\R^N} g(x,w_n)(\eta w_n)dx = o_n(1).
$$
By using $(g_5)$ and the properties of the function $\eta$,  
$$
\begin{aligned}
\int_{B^c(0,R)} |\nabla w_n|^2dx + \int_{\overline{B}(0,R)}w_n \nabla \eta \nabla w_ndx &+ \int_{B^c(0, \frac{R}{2})} (1+\lambda V(x))w_n^2dx\\
&\leq \frac{1}{k}\int_{B^c(0, \frac{R}{2})} (1+\lambda V(x))\eta w_n^2dx + o_n(1).
\end{aligned}
$$
As $k>2$, 
$$
\int_{B^c(0,R)}|\nabla w_n|^2dx + \frac{1}{2}\int_{B^c(0, \frac{R}{2})}(1+\lambda V(x)) \eta w_n^2dx \leq \int_{\overline{B}(0,R)} |w_n||\nabla \eta||\nabla w_n|dx + o_n(1),
$$
which leads to  
$$
\int_{B^c(0,R)} |\nabla w_n|^2dx + \frac{1}{2}\int_{B^c(0,R)} (1+\lambda V(x)) \eta w_n^2dx \leq \frac{M_1}{R}\int_{\overline{B}(0,R)} |w_n||\nabla w_n|dx + o_n(1).
$$
Now, employing by H\"older inequality there is a positive constant $C$ such that 
$$
\int_{B^c(0,R)}|\nabla w_n|^2\, dx + \int_{B^c(0,R)} (1+\lambda V(x))|w_n|^2\,dx  \leq \frac{C}{R} + o_n(1),
$$
and so
$$
\int_{B^c(0,R)} (|\nabla w_n|^2+|w_n|^2)dx \leq \frac{C}{R} + o_n(1).
$$
Now the lemma follows by taking $R$ large enough. 
\end{proof}

\begin{proposition}\label{bound3}
There exists $\hat{\lambda} = \hat{\lambda}(\tau)>0$ such that $I_{\lambda, \epsilon}$ satisfies the $(PS)_{c_{\lambda, \epsilon}}$ for any $c_{\lambda, \epsilon} \in (0, \frac{q-2}{2q}S^{N/2} - \tau)$ and for all $\lambda \geq \hat{\lambda}$, where $\tau$ is given by Lemma \ref{bound1}.
\end{proposition}
\begin{proof}
Let $(w_n) \subset E_{\lambda}$ be a sequence such that 
\begin{equation}\label{eq3.8}
I_{\lambda, \epsilon}(w_n) \to c_{\lambda, \epsilon}\quad \mbox{and}\quad I'_{\lambda, \epsilon}(w_n) \to 0\;\;\mbox{as}\;\;n\to +\infty.
\end{equation}
By Lemmas \ref{boundps} and \ref{nonneg}, $(w_n)$ is a bounded sequence that can be assumed non-negative, that is, $w_n \geq 0$ for all $n\in \N$. Then, there is $w\in E_\lambda$ such that, up to a subsequence,  
$$
\begin{aligned}
&w_n \rightharpoonup w\;\;\mbox{in}\;\;E_\lambda,\\
&w_n \to w\;\;\mbox{in}\;\;L_{loc}^{s}(\R^N), \;s\in [2,2^*)\;\;\mbox{and}\\
& w_n(x) \to w(x)\;\;\mbox{a.e. in $\R^N$}.
\end{aligned}
$$
A simple computation ensures that $I_{\lambda,\epsilon}'(w) = 0$.
\begin{claim}\label{claim1}
Let $(\nu_n)$ the sequence associated to $(w_n)$ that is obtained by Lemma \ref{C-C}. Then 
$$
\nu_n = 0, \quad \forall n\in \N.
$$
\end{claim}

\noindent 
In fact, let 
$$
\psi_\beta(x) =\psi \left(\frac{x-x_j}{\beta} \right),\;\;\forall x\in \R^N\;\;\mbox{and}\;\; \beta >0, 
$$
where $\psi \in C_{0}^{\infty}(\mathbb{R}^N)$ is such that  
$$
\begin{aligned}
&\psi (x) \in [0,1],\;\;\forall x\in \R^N;\\
&\psi (x) = 1,\;\;x\in B(0,1);\\
&\psi(x) = 0,\;\;x\in B^c(0,2)\;\;\mbox{and}\\
&|\nabla \psi(x)| \leq 2, \quad \forall x \in \mathbb{R}^N.
\end{aligned}
$$
It is standard to show that $(\psi_\beta w_n) \subset E_\lambda$ is bounded. Then  
%$$
%I'_{\lambda, \epsilon}(u_n)(\psi_\beta u_n) = o_n(1),%
%$$  
%that is 
$$
\int_{\R^N} [\nabla w_n \nabla (\psi_\beta w_n) + (1+\lambda V(x) )\psi_\beta w_n^2]dx - \frac{1}{\epsilon}\int_{\Omega}(\varphi - w_n)^+\psi_\beta w_ndx - \int_{\R^N} g(x,w_n)\psi_\beta w_ndx = o_n(1).
$$
This combined with the growth conditions on $g$ and the definition of $\psi_\beta$  gives  
$$
\begin{aligned}
\int_{\R^N} w_n\nabla w_n \nabla \psi_\beta dx + \int_{\R^N}(1+\lambda V(x))\psi_\beta w_n^2dx & \leq \mu \int_{\R^N} |w_n|^q \psi_\beta dx + \frac{1}{\epsilon}\int_{\Omega} (\varphi-w_n)^+\psi_\beta w_ndx\\
&+\int_{\R^N}|w_n|^{2^*}\psi_\beta dx - \int_{\R^N} |\nabla w_n|^2 \psi_\beta dx.
\end{aligned}
$$
Taking $n\to +\infty$ and using Lemma \ref{C-C},  we find 
\begin{equation}\label{eq3.9}
\limsup_{n\to +\infty}(L_n- M_n) \leq \mu \int_{\R^N} |w|^q \psi_\beta dx + \int_{\R^N}\psi_\beta d\nu - \int_{\R^N} \psi_\beta d\xi - \int_{\R^N} (1+\lambda V(x))\psi_\beta w^2dx,
\end{equation}
where
$$
L_n = \int_{\R^N} w_n\nabla w_n \nabla \psi_\beta dx\quad \mbox{and}\quad M_n = \frac{1}{\epsilon}\int_{\Omega}(\varphi - w_n)^+ \psi_\beta w_n dx.
$$
Now we are going to show $\displaystyle \lim_{\beta \to 0}(\limsup_{n\to \infty}|L_n|)=0$. To begin with, note that 
$$
L_n = \int_{\R^N} w_n\nabla w_n \nabla \psi_\beta dx = \int_{B(x_j, 2\beta)}w_n \nabla w_n \nabla \psi_\beta dx.
$$
Thereby, by H\"older inequality,  
$$
\begin{aligned}
|L_n| &\leq \int_{B(x_j, 2\beta)} |\nabla w_n| |\nabla \psi_\beta|w_ndx\\
&\leq |\nabla w_n|_2 \left( \int_{B(x_j, 2\beta)}|\nabla \psi_\beta|^2 |w_n|^2 dx \right)^{1/2} \leq \|w_n\| \left( \int_{B(x_j, 2\beta)} |\nabla \psi_\beta|^2|w_n|^2dx \right)^{1/2}.
\end{aligned}
$$
Now, the boundedness of $(u_n)$ in $H^1(\R^N)$ together with  the Sobolev embeedings and H\"older inequality yields 
\begin{equation}\label{eq3.10}
\begin{aligned}
\limsup_{n\to \infty} |L_n| &\leq M \left( \int_{B(x_j, 2\beta)} |\nabla \psi_\beta|^2|w|^2dx\right)^{1/2}\\
&\leq M\left( \int_{B(x_j, 2\beta)} |w|^{2^*}dx\right)^{1/2^*} \left( \int_{B(x_j, 2\beta)} |\nabla \psi_{\beta}|^Ndx\right)^{1/N},
\end{aligned}
\end{equation}
where $M= \sup_{n\in \N}\|w_n\|$. Now, by doing a change of variable $z= \beta y + x_j$, we have 
$$
\int_{B(x_j, 2\beta)}|\nabla \psi_{\beta}|^N(x)dx = \int_{B(0, 2)} |\nabla \psi_\beta|^N(z\beta + x_j)dz = \int_{B(0,2)} |\nabla \psi|^N(z)dz. 
$$
Since,
$$
\lim_{\beta \to 0}\int_{B(x_j, 2\beta)} |w|^{2^*}dx =0,
$$
the estimate (\ref{eq3.10}) ensures that  
\begin{equation}\label{eq3.11}
\lim_{\beta\to 0}\left(\limsup_{n\to \infty}|L_n|\right) = 0.
\end{equation}
In a similar way, 
\begin{equation}\label{eq3.12}
\lim_{\beta \to 0} \left( \limsup_{n\to \infty}|M_n|\right) =  0.
\end{equation}
Thereby, letting $\beta \to 0$ in (\ref{eq3.9}) and applying Lemma \ref{C-C}, we get 
\begin{equation}\label{eq3.13}
S\nu_i^{2/2^*}\leq \xi(x_j) \leq \nu_j,\;\;\forall j\in \N.
\end{equation}
If $\nu_j>0$, then 
\begin{equation}\label{eq3.14}
\nu_j \geq S^{N/2},\;\;\forall j\in \N.
\end{equation}
This combined with (\ref{C-C1}) guarantees that there is $j_0 \in  \mathbb{N}$ such that $\nu_j=0$ for all $j \geq j_0$. Now, we claim that $\nu_j = 0$ for all $j\in \N$. In fact, as $(w_n)$ is a $(PS)_{c_{\lambda, \epsilon}}$ sequence,  we must  have  
%$$
%I_{\lambda, \epsilon}(u_n) - \frac{1}{q}I'_{\lambda, \epsilon} (u_n)u_n \leq c_{\lambda, \epsilon} + o_n(1),
%$$
%that is
$$
\begin{aligned}
&\left(\frac{1}{2}-\frac{1}{q} \right)\int_{\R^N} |\nabla w_n|^2dx + \left( \frac{1}{2} - \frac{1}{q}\right)\int_{\R^N} (1+\lambda V(x))|w_n|^2dx \\
&+ \frac{1}{q}\int_{\R^N} (g(x,w_n)w_n - q G(x,w_n))dx = c_{\lambda, \epsilon} + o_n(1).
\end{aligned}
$$
As in the proof of Lemma \ref{boundps}, 
$$
\Big( \frac{1}{2} - \frac{1}{q}\Big) \int_{\R^N} (1+\lambda V(x))|w_n|^2dx + \frac{1}{q}\int_{\R^N}[g(x,w_n)w_n - qG(x,w_n)]dx \geq 0.
$$
Recalling that 
$$
\Big(\frac{1}{2} - \frac{1}{q} \Big) \int_{\R^N} \psi_\beta |\nabla w_n|^2dx \leq \Big(\frac{1}{2} - \frac{1}{q} \Big) \int_{\R^N} |\nabla w_n|^2dx,
$$
it follows that
$$
\Big(\frac{1}{2} - \frac{1}{q} \Big) \int_{\R^N} \psi_\beta |\nabla w_n|^2dx \leq c_{\lambda, \epsilon} + o_n(1).
$$
Therefore, taking $\beta \to 0$ in the last inequality we find
\begin{equation}\label{eq3.15}
\Big(\frac{1}{2} - \frac{1}{q} \Big) \xi(x_j) \leq c_{\lambda, \epsilon},\;\;\forall j\in \N.
\end{equation}
If there is $j_1\in \N$ such that $\nu_{j_1}>0$, by (\ref{eq3.13})-(\ref{eq3.15}),  
$$
c_{\lambda, \epsilon} \geq \Big(\frac{1}{2} - \frac{1}{q} \Big)\xi (x_{j_1}) \geq \Big(\frac{1}{2} - \frac{1}{q} \Big) S\nu_{j_1}^{1/2^*} \geq \Big(\frac{1}{2} - \frac{1}{q} \Big) S^{N/2}, %\Big(\frac{1}{2} - \frac{1}{q} \Big) S(S^{N/2})^{2/2^*}, 
$$
%namely 
%$$
%c_{\lambda, \epsilon} \geq \Big(\frac{1}{2} - \frac{1}{q} \Big) S^{N/2},
%$$
which contradicts the fact that   
$$
c_{\lambda, \epsilon} \in \Big(0, \frac{q-2}{2q}S^{N/2} - \tau \Big).
$$  
Thereby, $\nu_j = 0$ for all $j\in \N$. 

\noindent 
By Claim \ref{claim1}, it follows that $|w_n|^{2^*} \rightharpoonup |w|^{2^*}$ in $\mathcal{M}(\R^N)$, that is 
\begin{equation}\label{eq3.16}
\lim_{n\to \infty} \int_{\R^N} \psi|w_n|^{2^*}dx = \int_{\R^N}\psi |w|^{2^*}dx,\;\;\psi \in C_{0}^{\infty}(\R^N),
\end{equation}
which leads to 
\begin{equation}\label{eq3.17}
w_n \to w \;\;\mbox{in}\;\;L_{loc}^{2^*}(\R^N).
\end{equation}

Now, we are going to show that $w_n \to w$ in $E_{\lambda}$. Since $(w_n)$ is a $(PS)$ sequence,  
\begin{equation}\label{eq3.19-}
\|w_n - w\|_{\lambda}^2 = \int_{\R^N}g(x,w_n)w_ndx - \int_{\R^N} g(x,w_n)wdx + o_n(1).
\end{equation}
\begin{claim}\label{claim2} 
The limits below hold:  
\begin{enumerate}
\item $\displaystyle \lim_{n\to \infty} \int_{\R^N} g(x,w_n)w_ndx = \int_{\R^N} g(x,w)wdx$;
\item $\displaystyle \lim_{n\to \infty} \int_{\R^N}g(x,w_n)vdx = \int_{\R^N}g(x,w)vdx$, $\forall v\in E_\lambda$.
\end{enumerate}
\end{claim}

\noindent 
Assuming that the Claim \ref{claim2} occurs, by (\ref{eq3.19-}) we can infer that 
$$
\|w_n - w\|_\lambda^{2} = o_n(1),
$$ 
showing the $(PS)$ condition.

\noindent {\bf Proof of  Claim \ref{claim2}: }
$(1)$ Given $\delta>0$, consider $R>0$ as in Lemma \ref{vanish} and let 
$$
I_{n,1} = \int_{B(0,R)} |g(x,w_n)w_n - g(x,w)w|dx\quad \mbox{and}\quad I_{n,2} = \int_{B^c(0,R)} |g(x,w_n)w_n - g(x,w)w|dx.
$$
%By ($g_3$), we have 
%$$
%|g(x,t)t| \leq \mu |t|^2 + C_\beta|t|^{2^*},\;\;\forall (x,t)\in \R^N \times \R.
%$$
By Sobolev embedding and (\ref{eq3.17}),  
\begin{equation}\label{eq3.20}
\lim_{n\to \infty} I_{n,1} = 0.
\end{equation}  
On the other hand, since $w\in L^2(\R^N)$, we can fix $R>0$ large enough satisfying  
$$
\int_{B^c(0,R)} |w|^2dx \leq \frac{\delta}{2k}.
$$
By Lemma \ref{vanish}, 
$$
\limsup_{n\to +\infty} \Big( \int_{B^c(0,R)} (|\nabla w_n|^2+ |w_n|^2)dx  \Big) < \frac{\delta}{2k}.
$$
Moreover, as $\widetilde{\Omega} \subset B(0,R)$, then 
$$
|g(x,t)t|\leq \frac{1}{k}|t|^2,\;\;\forall x\in \R^N \setminus B(0,R),\;\forall t\in \R,
$$ 
from where it follows that 
$$
\begin{aligned}
 \limsup_{n\to +\infty}  |I_{n,2}| &\leq \limsup_{n\to +\infty} \left(\frac{1}{k}\int_{B^c(0,R)} |w_n|^2dx + \frac{1}{k}\int_{B^c(0,R)} |w|^2dx\right) \\
 &\leq \limsup_{n\to +\infty}  \left(\frac{1}{k} \int_{B^c(0,R)} (|\nabla w_n|^2 + |w_n|^2)dx + \frac{1}{k} \int_{B^c(0,R)} |w|^2dx\right) < \delta, \;\;\forall \delta >0,
\end{aligned}
$$
and so, 
\begin{equation}\label{eq3.21}
\limsup_{n\to +\infty} |I_{n,2}| =0.
\end{equation}
Now (1) follows from (\ref{eq3.20}) and (\ref{eq3.21}). (2) is proved of a similar way.  
\end{proof}
\begin{theorem}\label{pertur}
Under the hypotheses of Theorem \ref{main1}, there are $\mu_*, \widehat{\lambda}>0$ such that problem (\ref{eq3.5}) hast at least one weak solution $u_\epsilon$, for every $\lambda \geq \widehat{\lambda}$, $\mu \geq \mu_*$ and $\epsilon >0$.  
\end{theorem}
\begin{proof}
By Lemma \ref{bound1}, there is $\mu_*>0$ such that $c_{\lambda,\epsilon}<\frac{q-2}{2q}S^{N/2}-\tau$, for all $\mu\geq\mu_*$ and $\epsilon>0$.  Hence, from Proposition \ref{bound3} there exists $\widehat{\lambda}>0$ such that the functional $I_{\lambda,\epsilon}$ satisfies $(PS)_{c_{\lambda,\epsilon}}$ condition for all $\lambda\geq\widehat{\lambda}$. Then, by mountain pass theorem \cite{AR}, there is $u_\epsilon\in E_\lambda$ such that
$$
I_{\lambda,\epsilon}(u_\epsilon)=c_{\lambda,\epsilon} \quad \text{and} \quad I'_{\lambda,\epsilon}(u_\epsilon)=0.
$$
\end{proof}

%%%%%%%%%%%%%%%%%%%%%%%%%%%%%%%%%%%%%%%%%%%%%%%%%

\subsection{Solution to the modified variational inequality}

For $\lambda \geq \lambda_*$ and $\mu \geq \mu_*$ fixed, there is a non-trivial and non-negative solution $u_\epsilon \in E_\lambda$ of the problem (\ref{eq3.5}), that is, $u_\epsilon \geq 0$ and  
$$
\int_{\R^N} [\nabla u_\epsilon \nabla v + (1+\lambda V(x)) u_\epsilon v]dx + \frac{1}{\epsilon}\langle P(u_\epsilon), v \rangle = \int_{\R^N} g(x,u_\epsilon)vdx,\;\;\forall v\in E_\lambda.
$$ 
In what follows, let us consider the following notations 
\begin{equation}\label{eq3.22-Z}
\epsilon = \frac{1}{n},\quad u_n = u_{\frac{1}{n}},\quad I_n = I_{\lambda, \epsilon}\quad \mbox{and}\quad I_n(u_n) = c_n= c_{\lambda, \epsilon}.
\end{equation}
Then, for each $n\in \N$ there is $u_n \in E_\lambda$ such that 
\begin{equation}\label{eq3.22-}
\int_{\R^N} [\nabla u_n \nabla v + (1+\lambda V(x))u_nv]dx + n\langle P(u_n), v\rangle = \int_{\R^N} g(x,u_n)vdx,\;\;\forall v\in E_\lambda.
\end{equation}
Moreover, as in Lemma \ref{boundps}, $(u_n)$ is bounded in $H^1(\R^N)$ and there is $u\in H^1(\R^N)$ such that, up to a subsequence,  
$$
\begin{aligned}
&u_n \rightharpoonup u\;\;\mbox{in}\;\;E_\lambda,\\
&u_n \to u\;\;\mbox{in}\;\;L_{loc}^{s}(\R^N), \;\;s\in [2, 2^*),\\
&u_n(x) \to u(x),\;\;\mbox{a.e. in}\;\;\R^N.
\end{aligned}
$$

Reasoning as in \cite[Lemma 3.11]{CALBCT}, it is possible to prove that $P(u) = 0$, that is, $u\in \mathbb{K}$.  

\begin{lemma}\label{conver}
The following convergence is true 
$$
u_n \to u\;\;\mbox{in}\;\;E_\lambda.
$$
\end{lemma}
\begin{proof}
Since
\begin{equation}\label{eq3.24-}
\|u_n\|_{\lambda}^{2} = I'_n(u_n)u_n - n \langle P(u_n), u_n\rangle + \int_{\R^N}g(x,u_n)u_n dx
\end{equation} 
and 
\begin{equation}\label{eq3.25-}
-\langle u_n, u\rangle_\lambda = -I'_{n}(u_n)u + n\langle P(u_n), u \rangle - \int_{\R^N} g(x,u_n)udx,
\end{equation}
we obtain 
\begin{equation}\label{eq3.26-}
\|u_n - u\|_{\lambda}^{2} = n\langle P(u_n), u-u_n\rangle + \int_{\R^N}g(x,u_n)u_ndx - \int_{\R^N}g(x,u_n)udx + o_n(1).
\end{equation}
Recalling that 
$$
\langle P(u_n), u-u_n\rangle = \langle P(u_n) - P(u), u-u_n\rangle \leq 0,
$$
the inequality (\ref{eq3.26-}) gives  
\begin{equation}\label{eq3.27-}
\|u_n-u\|_{\lambda}^{2} \leq \int_{\R^N} g(x,u_n)u_ndx - \int_{\R^N} g(x,u_n)udx + o_n(1).
\end{equation}

As in  Lemma \ref{vanish}, given $\delta >0$, there is $R>0$ such that 
\begin{equation}\label{eq3.30}
\limsup_{n\to \infty}\left(\int_{B^c(0,R)}(|\nabla u_n|^2 + |u_n|^2)dx \right) < \delta.
\end{equation}

Our next step is to show that Claim \ref{claim1} also holds for the sequence $(u_n)$ fixed in (\ref{eq3.22-Z}). Again, let us consider 
$$
\psi_\beta(x) = \psi \left( \frac{x-x_j}{\beta}\right)\;\;\mbox{for all}\;\;x\in \R^N\;\;\mbox{and for all} \;\;\beta >0,
$$ 
where $\psi \in C_{0}^{\infty}(\mathbb{R}^N)$ is such that $\psi = 1$ in $B(0,1)$, $\psi = 0$ in $B^c(0,2)$ and $|\nabla \psi(x)| \leq 2$ for all $x \in \mathbb{R}^N$, with $0\leq \psi \leq 1$. Note that 
%$$
%I'_n(u_n)(\psi_\beta u_n - \psi_\beta \varphi^+) = 0,
%$$
%namely, 
\begin{equation}\label{eq3.33}
\begin{aligned}
\int_{\R^N}[\nabla u_n \nabla (\psi_\beta u_n - \psi_\beta \varphi^+) &+ (1+\lambda V(x))u_n (\psi_\beta u_n - \psi_\beta \varphi^+)]dx \\
&+ n \langle P(u_n), \psi u_n - \psi \varphi^+\rangle = \int_{\R^N} g(x,u_n)(\psi_\beta u_n - \psi_\beta \varphi^+)dx.
\end{aligned}
\end{equation}
Thus, 
\begin{equation}\label{eq3.34-}
\int_{\R^N}[\nabla u_n \nabla (\psi_\beta u_n - \psi_\beta \varphi^+) + (1+\lambda V(x))u_n (\psi_\beta u_n - \psi_\beta \varphi^+)]dx \leq \int_{\R^N}g(x,u_n)(\psi_\beta u_n - \psi_\beta \varphi^+)dx.
\end{equation}
Letting 
$$
L_{n,\beta} = \int_{\R^N}[\nabla u_n \nabla (\psi_\beta \varphi^+) + (1+\lambda V(x))u_n \psi_\beta \varphi^+]dx - \int_{\R^N}g(x,u_n)\psi_\beta \varphi^+dx,
$$
(\ref{eq3.34-}) can be rewrite as 
\begin{equation}\label{eq3.35-}
\int_{\R^N} [\nabla u_n \nabla (\psi_\beta u_n) + (1+\lambda V(x))u_n \psi_\beta u_n]dx - L_{n,\beta} \leq \int_{\R^N}g(x,u_n)\psi_\beta u_n dx.
\end{equation}
As before, it is possible to show that 
\begin{equation}\label{eq3.36-}
\lim_{\beta \to 0}\left( \limsup_{n\to \infty} |L_{n,\beta}| \right) = 0.
\end{equation}
Furthermore, as in the proof of Claim \ref{claim1},  $\nu_j =0$ for all $j\in \N$. Consequently 
\begin{equation}\label{eq3.37-}
u_n \to u\;\,\mbox{in}\;\;L_{loc}^{2^*}(\R^N).
\end{equation}

This limit together with ( \ref{eq3.30}) leads to  
\begin{equation}\label{eq3.28-}
\lim_{n\to \infty} \int_{\R^N}g(x,u_n)u_ndx = \int_{\R^N}g(x,u)udx
\end{equation} 
and 
\begin{equation}\label{eq3.29-}
\lim_{n\to \infty}\int_{\R^N}g(x,u_n)vdx = \int_{\R^N}g(x,u)vdx,\;\;\forall v\in H^1(\R^N).
\end{equation}
All of information are sufficient to conclude that 
$$
\|u_n - u\|_\lambda^2 = o_n(1),
$$
finishing the proof of Lemma \ref{conver}.
\end{proof}

Since   
$$
\langle P(u_n), w-u_n \rangle  = \langle P(u_n) - P(w), w-u_n\rangle \leq 0, \quad \forall w\in \mathbb{K}
$$
taking $v = w-u_n$ as a test function in (\ref{eq3.22-}), we obtain 
\begin{equation}\label{eq3.38-}
\int_{\R^N} \nabla u_n \nabla (w-u_n)dx + \int_{\R^N}(1+\lambda V(x))u_n (w-u_n)dx \geq \int_{\R^N} g(x,u_n)(w-u_n)dx.
\end{equation}
Thus, letting $n\to +\infty$ in (\ref{eq3.38-}) and employing Lemma \ref{conver},  
\begin{equation}\label{eq3.39}
\int_{\R^N}\nabla u \nabla (w-u)dx + \int_{\R^N}(1+\lambda V(x))u(w-u)dx \geq \int_{\R^N}g(x,u)(w-u)dx,
\end{equation}
showing that $u$ is a non-negative weak solution for problem (\ref{eq3.2}). 

%%%%%%%%%%%%%%%%%%%%%%%%%%%%%%%%%%%%%%%%%%%%%%%%%
 
\subsection{Proof of the Theorem \ref{main1}}

By the last section, we found a solution $u_\lambda$ to the modified variational inequality (\ref{eq3.2}). Now, we must show that $u_\lambda$ is a solution to the variational inequality (\ref{eq3.1}). In order to prove this, it is enough to show that  
\begin{equation}\label{eq3.40-}
u_\lambda (x) \leq a,\;\;\forall x\in \R^N \setminus \tilde{\Omega}.
\end{equation}

In what follows, let us consider $u_n\in E_{\lambda_n}$ and $\lambda_n \to \infty$ satisfying 
\begin{equation}\label{eq3.41-}
\int_{\R^N}\nabla u_n \nabla (v-u_n)dx + \int_{\R^N}(1+\lambda_n V(x))u_n (v-u_n)dx \geq \int_{\R^N}g(x,u_n)(v-u_n)dx, \quad \forall v\in \mathbb{K}.
\end{equation}

\begin{proposition}\label{propo}
Let $(u_n)$ be a sequence satisfying (\ref{eq3.41-}). Then, there exists a subsequence of $(u_n)$, still denoted by $(u_n)$, and $u\in H^1(\R^N)$ such that 
$$
u_n \rightharpoonup u \;\;\mbox{in}\;\;H^1(\R^N).
$$
Furthermore, 
\begin{enumerate}
\item $u \equiv 0$ in $\R^N \setminus \Omega$.
\item $\|u_n- u\|_{\lambda_n}^2 \to 0$.
\item As $\lambda_n \to \infty$ we have the following limits: 
$$
\begin{aligned}
&u_n \to u\;\;\mbox{in}\;\;H^1(\R^N),\\
& \lambda_n \int_{\R^N}V(x)|u_n|^2dx \to 0,\\
& \|u_n\|_{\lambda_n}^{2} \to \int_{\Omega}(|\nabla u|^2 + |u|^2)dx = \|u\|_{H^1(\Omega)}^{2}.
\end{aligned}
$$
\item The function $u$ is a solution of the variational inequality 
\begin{equation}\label{eq3.42-}
\int_{\Omega}\nabla u \nabla (v-u)dx + \int_{\Omega}u(v-u)dx \geq \int_{\Omega} (\mu |u|^{q-2}+|u|^{2^*-2})u(v-u)dx, 
\end{equation}
for every $v\in \tilde{\mathbb{K}}$, where 
$$
\tilde{\mathbb{K}} = \{v\in H_{0}^{1}(\Omega); v\geq \varphi \;\;\mbox{a.e. in}\;\;\Omega\}.
$$
\end{enumerate}
\end{proposition}
\begin{proof}
As in Lemma \ref{bound2},  
\begin{equation}\label{eq3.43-} 
\limsup_{n\to +\infty} \|u_n\|_{\lambda_n}^{2} \leq \frac{4qM_2}{q-2},
\end{equation}
where $M_2=\sup_{n\in\N}c_{\lambda_n}$, which implies that $(\|u_n\|_{\lambda_n})$ is bounded in $\R$. Since 
$$
\|u_n\|_{\lambda_n} \geq \|u_n\|,\;\;\forall n\in \N,
$$
$(u_n)$ is also bounded in $H^1(\R^N)$. Therefore, there is $u\in H^1(\R^N)$ such that, up to a subsequence, 
$$
u_n \rightharpoonup u \;\;\mbox{in}\;\;H^1(\R^N).
$$ 
Now, we are going to show the items $(1)-(4)$.
\begin{enumerate}
\item It follows in a similar way as in \cite[Proposition 3.12]{CALB}. 

\item First of all, note that 
\begin{equation}\label{eq3.44-}
\|u_n - u\|_{\lambda_n}^{2} = \langle u_n, u_n-u\rangle_{\lambda_n} -  \langle u, u_n-u\rangle_{\lambda_n}.
\end{equation}
Since $u_n \rightharpoonup u$ in $H^1(\R^N)$, by (1),
$$
\begin{aligned}
 \langle u, u_n-u\rangle_{\lambda_n} %&= \int_{\R^N} [\nabla u \nabla (u_n-u) + (1+\lambda_n V(x)) u (u_n-u)]dx\\
&= \int_{\Omega} [\nabla u \nabla (u_n-u) + u(u_n-u)]dx = o_n(1).
\end{aligned}
$$
So, (\ref{eq3.44-}) can be rewritten as
\begin{equation}\label{eq3.45-}
\|u_n - u\|_{\lambda_n}^{2} = \langle u_n, u_n-u\rangle_{\lambda_n} + o_n(1).
\end{equation} 
On the other hand, as $u_n \geq \varphi$ a.e. in $\Omega$ for all $n\in \N$, then $u\geq \varphi$ a.e. in $\Omega$, and so, consequently $u\in \mathbb{K}$. Then, we can take $u$ as a test function in (\ref{eq3.41-}) to get 
\begin{equation}\label{eq3.46-}
%\int_{\R^N}\nabla u_n \nabla (u_n - u)dx + \int_{\R^N}(1+\lambda_n V(x))u_n (u_n-u)dx \leq \int_{\R^N}g(x,u_n)(u_n-u)dx 
\langle u_n, u_n-u\rangle_{\lambda_n} \leq \int_{\R^N} g(x,u_n)(u_n-u)dx.
\end{equation}
Combining (\ref{eq3.45-}) with (\ref{eq3.46-}), 
\begin{equation}\label{eq3.47-}
\|u_n - u\|_{\lambda_n}^{2} \leq \int_{\R^N}g(x,u_n)(u_n-u)dx + o_n(1).
\end{equation}
In order to conclude the proof, we must prove
\begin{equation}\label{eq3.48-}
\lim_{n\to \infty}\int_{\R^N}g(x,u_n)u_ndx = \int_{\R^N}g(x,u)u\,dx
\end{equation} 
and 
\begin{equation}\label{eq3.49-}
\lim_{n\to \infty}\int_{\R^N}g(x,u_n)vdx = \int_{\R^N}g(x,u)vdx,\;\;\forall v\in H^1(\R^N).
\end{equation}
However, as it was done in the previous section, it is enough to show the limit
\begin{equation}\label{eq3.50-}
u_n \to u\;\;\mbox{in}\;\;L_{loc}^{2^*}(\R^N),
\end{equation}
and that given $\delta >0$, there is $R>0$ such that 
\begin{equation}\label{eq3.51} 
\limsup_{n\to +\infty} \left(\int_{B^c(0,R)} (|\nabla u_n|^2 + |u_n|^2)dx \right) < \delta.
\end{equation}
 
 As in the proof of Proposition \ref{bound3}, (\ref{eq3.50-}) and (\ref{eq3.51}) are fundamental to justify the limits (\ref{eq3.48-}) and (\ref{eq3.49-}). To show (\ref{eq3.51}) we follow the same steps given in the proof of Lemma \ref{vanish}. The main difficulty to show (\ref{eq3.50-}) is to prove that $\nu_j=0$ for all $j\in \N$. So, we need to choose a suitable test function in (\ref{eq3.41-}), and thereby establishing an inequality similar to (\ref{eq3.34-}). In fact, as above we will consider 
$$
\psi_\beta (x) = \psi \Big( \frac{x-x_j}{\beta} \Big)\;\;\forall x\in \R^N\;\;\mbox{and}\;\;\forall \beta >0,
$$     
where $\psi\in C_{0}^{\infty}(\mathbb{R}^N)$ is such that $\psi = 1$ in $B(0,1)$, $\psi = 0$ in $B^c(0,2)$ and $|\nabla \psi| \leq 2$ with $0\leq \psi \leq 1$.
\begin{claim}\label{claim4}
Let $v_n = u_n - \psi_\beta (u_n- \varphi^+)$. Then $v_n \in \mathbb{K}$.
\end{claim}

\noindent 
In fact, we know that $u_n \geq \varphi^+ \geq \varphi,$ for all $n\in \N$. Note that, if $\psi_\beta = 0$, then $v_n = u_n \in \mathbb{K}$. Similarly, if $\psi_\beta = 1$, we have $v_n = \varphi^+ \in \mathbb{K}$. Now we suppose that $0< \psi_\beta <1$. Observe that 
$$
u_n - \psi_\beta u_n = u_n (1-\psi_\beta) \geq \varphi^+ (1-\psi_\beta) = \varphi^+ - \varphi^+ \psi_\beta,
$$
consequently 
$$
v_n = u_n - \psi_\beta u_n + \varphi^+  \psi_\beta \geq \varphi^+,
$$
next $v_n \in \mathbb{K}$ and we finish the proof of Claim \ref{claim4}.

By Claim (\ref{claim4}), we can take $v_n$ as a test function in (\ref{eq3.41-}) to obtain 
\begin{equation}\label{eq3.52}
\begin{array}{l} 
\displaystyle \int_{\R^N}[\nabla u_n \nabla (\psi_\beta u_n - \psi_\beta \varphi^+) + (1+\lambda_n V(x))u_n (\psi_\beta u_n - \psi_\beta \varphi^+)]dx \\
\mbox{} \\
\hspace{8 cm} \leq \displaystyle \int_{\R^N} g(x,u_n) (\psi_\beta u_n - \psi_\beta \varphi^+)dx. 
\end{array}
\end{equation}
Note that inequality (\ref{eq3.52}) looks similar to (\ref{eq3.34-}), then, following the same steps as in the proof of Claim \ref{claim1} we can infer that $\nu_j = 0$ for all $j\in \N$, which proves (\ref{eq3.50-}). Therefore, as in the proof of Lemma \ref{conver},  
$$
\|u_n - u\|_{\lambda_n}^{2} = o_n(1),
$$    
and the proof is complete.

\item Note that 
$$
\|u_n -u\|^2 \leq \|u_n -u\|_{\lambda_n}^{2}
$$
and 
$$
\int_{\R^N} \lambda_n V(x)|u_n|^2dx = \int_{\R^N} \lambda_n V(x)|u_n - u|^2dx \leq C \|u_n - u\|_{\lambda_n}^{2}.
$$
So, by item $(2)$,  
$$
\|u_n - u\|^2 \to 0
$$
and 
$$
\int_{\R^N}\lambda_n V(x)|u_n|^2dx \to 0.
$$

\item Let $v\in \tilde{\mathbb{K}}$ be arbitrary. Then,  by (\ref{eq3.41-}), 
$$
\int_{\R^N} \nabla u_n\nabla (v-u_n)dx + \int_{\R^N}u_nvdx- \int_{\R^N}(1+\lambda_n V(x))|u_n|^2dx \geq \int_{\R^N}g(x,u_n)(v-u_n)dx.
$$ 
Taking $n\to +\infty$ and using $(1)-(3)$, 
$$
\int_{\Omega}[\nabla u \nabla (v-u) + u(v-u)]dx \geq \int_{\Omega}g(x,u)(v-u)dx,\;\;\forall v\in \tilde{\mathbb{K}}.
$$
\end{enumerate}
\end{proof}

\noindent 
{\bf Proof of the Theorem \ref{main1}(conclusion):} Finally we are ready to conclude the proof of the Theorem \ref{main1}, and now the main tool is the iterative Moser method \cite{JM} based on the work of Gongbao \cite{LG}, which is a novelty for this type of inequality. For each $n\in \N$ and $L\geq 1$, let us define 
$$
u_{n,L}(x) = \begin{cases}
u_n,&\mbox{if $u_n \leq L$}\\
L, & \mbox{if $u_n \geq L$}.
\end{cases}
$$  
Let $d = dist(\overline{\Omega}, \partial \widetilde{\Omega})$ and $R\in (0,d)$. Given $x_0 \in \overline{\widetilde{\Omega}}^c$, we fixed a sequence $(r_j) \subset \R$ such that $R< r_j < d$, $\forall j\in \N$ and $r_j \searrow R$. Moreover, let 
$$
v_{n,L} (x) = v_{n,L,j}(x) = \left(u_n - u_n \eta^2 \left(\frac{u_{n,L}}{L} \right)^{2(\beta - 1)} \right)(x),
$$
where $\eta = \eta_j \in C^\infty (\R)$ is a function such that $0\leq \eta \leq 1$, $|\nabla \eta| \leq \frac{1}{r_j}$ and 
$$
\eta (x) = \begin{cases}
1,&\mbox{if $x\in B(x_0, r_{j+1})$}\\
0,&\mbox{if $x\in B^c(x_0 , r_j)$},
\end{cases}
$$
and $\beta >1$ will be fixed.
\begin{claim}\label{claim5}
For each $n\in \N$, $v_{n,L} \geq \varphi$ in $\R^N$.
\end{claim}

\noindent 
In fact, note that $\eta=\eta_j=0$ in $B^c(x_0, r_j)$, for all $j\in\N$ then 
$$
v_{n,L}=u_n\geq\varphi.
$$
Since $\varphi\leq 0$ in $\Theta^c$ and $B^c(x_0, {r_{j}})\subset\Theta^c$, for every $j\in\N$ its follows that $\varphi\leq 0$ in $B^c(x_0, {r_{j}}).$

On the other hand, as $u_{n,L},u_{n,L}^{\beta-1}>0$ and $-\eta^2\geq-1$, then  
$$
-u_n\eta^2\Big(\displaystyle\frac{u_{n,L}}{L}\Big)^{2(\beta-1)}\geq-u_n\Big(\displaystyle\frac{u_{n,L}}{L}\Big)^{2(\beta-1)},
$$ 
so
$$
v_{n,L}=u_n-u_n\eta^2\Big(\displaystyle\frac{u_{n,L}}{L}\Big)^{2(\beta-1)}\geq u_n\Big[1-\Big(\displaystyle\frac{u_{n,L}}{L}\Big)^{2(\beta-1)}\Big]\geq0,
$$
from where $v_{n,L}\geq0\geq\varphi$ in $B^c(x_0, {r_{j}})$. Therefore, $v_{n,L}\in\mathbb{K}$.

By Claim \ref{claim5},  $v_{n,L}\in\K$ for each $n\in\N$. Now, using $v_{n,L}$ as a test function, we obtain 
$$
\begin{array}{ccc}
\displaystyle\frac{1}{L^{2(\beta-1)}}\Big(\int_{\R^N}\nabla u_n\nabla(u_n\eta^2u_{n,L}^{2(\beta-1)})\,dx\Big)&+&\displaystyle\frac{1}{L^{2(\beta-1)}}\Big(\int_{\R^N}(1+\lambda_nV(x))u_n^2\eta^2u_{n,L}^{2(\beta-1)}\,dx\Big)\\
&\leq&\displaystyle\frac{1}{L^{2(\beta-1)}}\Big(\int_{\R^N} g(x,u_n)u_n\eta^2u_{n,L}^{2(\beta-1)}\,dx\Big)\\
\end{array}
$$
that is 
\begin{equation}\label{eq3.53}
\begin{array}{ccc}
\displaystyle\int_{\R^N}\nabla u_n\nabla(u_n\eta^2u_{n,L}^{2(\beta-1)})\,dx&\leq-\displaystyle\int_{\R^N}(1+\lambda_nV(x))u_n^2\eta^2u_{n,L}^{2(\beta-1)}\,dx +\displaystyle\int_{\R^N} g(x,u_n)u_n\eta^2u_{n,L}^{2(\beta-1)}\,dx.\hspace{0.7cm}
\end{array}
\end{equation}
Recalling that $\lambda_n V(x)\geq0$ for all $x\in\mathbb{R}^N$, it follows that
\begin{equation}\label{lim.1}
\displaystyle\int_{\R^N}\nabla u_n\nabla(u_n\eta^2u_{n,L}^{2(\beta-1)})\,dx\leq \displaystyle\int_{\R^N} g(x,u_n)u_n\eta^2u_{n,L}^{2(\beta-1)}\,dx
-\displaystyle\int_{\R^N}u_n^2\eta^2u_{n,L}^{2(\beta-1)}\,dx.
\end{equation}
Hence,   
\begin{equation}
\begin{array}{ccc}
\displaystyle\int_{\R^N}|\nabla u_n|^2\eta^2u_{n,L}^{2(\beta-1)}\,dx+2\displaystyle\int_{\R^N}\nabla u_n\nabla\eta u_n\eta u_{n,L}^{2(\beta-1)}\,dx
&\leq&\displaystyle\int_{\R^N} g(x,u_n)u_n\eta^2u_{n,L}^{2(\beta-1)}\,dx\\
&-&\displaystyle\int_{\R^N}u_n^2\eta^2u_{n,L}^{2(\beta-1)}\,dx.\hspace{2cm}
\end{array}
\end{equation}
Now, by $(g_3)$,
\begin{equation}
\begin{array}{ccc}
\displaystyle\int_{\R^N}|\nabla u_n|^2\eta^2u_{n,L}^{2(\beta-1)}\,dx+2\displaystyle\int_{\R^N}\nabla u_n\nabla\eta u_n\eta u_{n,L}^{2(\beta-1)}\,dx
&\leq&C_\beta\displaystyle\int_{\R^N} u_n^{2^*}\eta^2u_{n,L}^{2(\beta-1)}\,dx\\
&+&(\mu\beta-1)\displaystyle\int_{\R^N}u_n^2\eta^2u_{n,L}^{2(\beta-1)}\,dx.\\
\end{array}
\end{equation}
Fixing $0<\beta<1/\mu$,  we derive that  
\begin{equation}
\displaystyle\int_{\R^N}|\nabla u_n|^2\eta^2u_{n,L}^{2(\beta-1)}\,dx+2\displaystyle\int_{\R^N}\nabla u_n\nabla\eta u_n\eta u_{n,L}^{2(\beta-1)}\,dx
\leq C\displaystyle\int_{\R^N}u_n^{2^{*}}\eta^2u_{n,L}^{2(\beta-1)}\,dx,
%\label{lim.4}
\end{equation}
where $C$ be a positive constant. 

The above inequality permits to apply the same arguments found in \cite{LG} to deduce that 
$u_n\in L^\infty(B(x_0, R))$ and 
\begin{equation}
|u_{n}|_{L^{\infty}(B_{R}(x_0))}\leq C
|u_{n}|_{L^{2^*}(B_{r_{1}}(x_0))}, \ \ \forall n\in\N,
\end{equation}
where $C$ is independent of $n$. Moreover, as $B(x_0, r_1)\subset\Omega^c$,  
$$
|u_{n}|_{L^{\infty}(B(x_0, R))}\leq C
|u_{n}|_{L^{2^*}(B(x_0, r_1))}\leq C |u_{n}|_{L^{2^*}(\Omega^c)}, \ \ \forall n\in\N.
$$
Recalling that 
$$
u_n=u_{n,\lambda_n}\rightarrow0 \quad \mbox{in} \quad H^1(\Omega^c) \quad \mbox{as} \quad \lambda_n\rightarrow+\infty, 
$$
given $\kappa>0$ there is $n_0\in\N$ such that
$$
|u_{n}|_{L^{\infty}(B(x_0, R))}<\kappa, \ \ \text{for} \ n\geq n_0.
$$
As $x_0\in\tilde{\Omega}^c$ is arbitrary, we derive  that  
$$
|u_{n}|_{L^{\infty}(\tilde{\Omega}^c)}<\kappa, \quad \forall n \geq n_0. 
$$
Therefore, by the previous analysis, there is $\lambda^*>0$ such that
\begin{center}
	$u_\lambda(x)\leq a$, $\forall x\in\Omega^c$ and $\lambda\geq\lambda^*$.
\end{center}
%By (\ref{eq3.15}) we have 
%$$
%\displaystyle\int_{\mathbb{R}^N}\nabla u \nabla(v-u)\hspace{0.05cm}dx+\displaystyle\int_{\mathbb{R}^N}(1+\lambda V(x))u(v-u)\hspace{0.05cm}dx 
%\geq \int_{\mathbb{R}^N}g(x,u)(v-u)\hspace{0.05cm}dx, \quad \forall v \in \mathbb{K}. 
%$$
Thereby, for $\lambda\geq\lambda^*$, the function $u_\lambda$ satisfies the inequality below 
\begin{equation}
\displaystyle\int_{\mathbb{R}^N}\nabla u \nabla(v-u)\hspace{0.05cm}dx+\displaystyle\int_{\mathbb{R}^N}(1+\lambda V(x))u(v-u)\hspace{0.05cm}dx 
\geq \int_{\mathbb{R}^N}(\mu u^{q-2}+u^{2^*-2})u(v-u)\hspace{0.05cm}dx, \quad \forall v \in \mathbb{K},
\label{r5}
\end{equation}	
finishing  the proof of Theorem \ref{main1} for the case $N\geq 3$.

%%%%%%%%%%%%%%%%%%%%%%%%%%%%%%%%%%%%%%%%%%%%%%%%%

%%%%%%%%%%%%%%%%%%%%%%%%%%%%%%%%%%%%% O caso N=2 %%%%%%%%%%%%%%%%%%%%%%%%%%%%%%%%%%%%%%%%%

\section{Proof of Theorem \ref{main1}: Case $N=2$}

In this section, we follow the same strategy of the Section 3. From now on, let us fix  
$k=2\theta/(\theta-2)>2$ and $a>0$ such that $f(a)=a/k$. Moreover, we define the following functions 
$$ \widetilde{h}(t)= \left\{
\begin{array}{ll}
f(t),\ \ t \leq a \vspace{0.3cm}\\ 
\displaystyle\frac{t}{k},\ \ t \geq a\\
\end{array}
\right.
$$
and 
$$
\widetilde{H}(t)=\displaystyle\int_0^th(s)ds.
$$
Let $\widetilde{\Omega}$ be an open bounded connected set with smooth boundary such that $\overline{\Omega} \subset \widetilde{\Omega}$, and consider the functions $\widetilde{g}:\mathbb{R}^2\times\mathbb{R}\rightarrow \mathbb{R}$ and $\widetilde{G}:\mathbb{R}^2\times\mathbb{R}\rightarrow \mathbb{R}$ given by 
$$
\widetilde{g}(x,t)=\chi_{\widetilde{\Omega}}(x)\tilde{f}(t)+(1-\chi_{\widetilde{\Omega}}(x))h(t)
$$
and
$$
\widetilde{G}(x,t)=\chi_{\widetilde{\Omega}}(x)\tilde{F}(t)+(1-\chi_{\widetilde{\Omega}}(x))H(t),
$$
where
$$
\tilde{F}(t)=\displaystyle\int_0^t\tilde{f}(s)ds.
$$
Now, we deal with the existence of solution to the following modified variational inequality: 
\begin{equation}
\begin{cases}
u\in\mathbb{K},&\\
\displaystyle\int_{\mathbb{R}^2}\nabla u \nabla(v-u)\hspace{0.05cm}dx+\displaystyle\int_{\mathbb{R}^2}(1+\lambda V(x))u(v-u)\hspace{0.05cm}dx 
\geq \int_{\mathbb{R}^2}\widetilde{g}(x,u)(v-u)\hspace{0.05cm}dx, \quad \forall v \in \mathbb{K}. &
\end{cases}
\label{modificado2}
\end{equation}
Problem (\ref{modificado2}) is closely related to problem (\ref{eq3.2}) with $N=2$, because, if $u_\lambda$ is a solution of  (\ref{modificado2}) such that 
$$
u(x) \leq a, \ \forall x \in \mathbb{R}^2\setminus\widetilde{\Omega},
$$
then $u$ be a solution of the original variational inequality (\ref{eq3.1}). 

Next, we list some properties of the function $\widetilde{g}$. Since $a>0$, then 
\begin{itemize}
\item [$(\widetilde{g}_1)$] $\widetilde{g}(x,t)=0$, for all $x\in\mathbb{R}^2$ and $t\leq0$;
\item [$(\widetilde{g}_2)$] $\displaystyle\frac{\widetilde{g}(x,t)}{t}\rightarrow0$ as $t\rightarrow0$, uniformly in $x\in\mathbb{R}^2$;
\item [$(\widetilde{g}_3)$] Fixed $l>1$ and $\alpha\geq\alpha_0$ there is $C_\nu>0$ with $C_\nu \to 0$ as $\nu \to +\infty$ such that 
$$
|\widetilde{g}(x,t)|\leq \frac{1}{2}|t|+C_\nu|t|^{l}\Big(e^{\alpha|t|^2-1} \Big), \quad \forall (x,t)\in\mathbb{R}^2\times\mathbb{R}
$$
\item [$(\widetilde{g}_4)$] For all $x\in\widetilde{\Omega}$ and $t>0$ we have
$$
0<\theta \widetilde{G}(x,t)\leq \widetilde{g}(x,t)t;
$$
\item [$(\widetilde{g}_5)$] For all $x\in\mathbb{R}^2\setminus \widetilde{\Omega}$ and $t>0$ we have
$$
0<2\widetilde{G}(x,t)\leq \widetilde{g}(x,t)t\leq\frac{1}{k}(1+\lambda V(x))t^2.
$$
\end{itemize}

As in the last section, we are going to work with a penalized problem, that is, we will look for solutions to the following class of elliptic equation
\begin{equation}
-\Delta u+(1+\lambda V(x))u-\displaystyle\frac{1}{\epsilon}(\varphi-u)^+\chi_\Omega=\widetilde{g}(x,u), \ \ \text{in} \ \ \mathbb{R}^2.
\label{penalizado2}
\end{equation}
The weak solutions of (\ref{penalizado2}) are critical points of the functional $I_{\lambda,\epsilon}:
E_\lambda\rightarrow\mathbb{R}$ defined as
$$
I_{\lambda,\epsilon}(u)=\frac{1}{2}||u||^2_\lambda+\frac{1}{2\epsilon}\displaystyle\int_{\Omega}[(\varphi-u)^{+}]^2\,dx-\int_{\mathbb{R}^2}\widetilde{G}(x,u)\,dx,
$$
which is $C^1$ with 
$$
I_{\lambda,\epsilon}'(u)v=\langle u,v\rangle_\lambda-\frac{1}{\epsilon} \int_{\Omega}(\varphi-u)^+v\,dx- \int_{\mathbb{R}^2}\widetilde{g}(x,u) v\,dx,\;\;\forall u,v \in E_\lambda.
$$

In this section a key tool is the following Trundinger-Moser inequality due to Cao \cite{Cao}: 
$$
\displaystyle\int_{\mathbb{R}^2}\Big(e^{\alpha|u|^2}-1\Big)\hspace{0.05cm}dx<\infty, \ \forall u\in H^1(\mathbb{R}^2) \ \text{and} \ \alpha>0.
$$
Moreover, if $\alpha<4\pi$ and $|u|_2\leq M$, there exists $C=C(\alpha,M)>0$ such that
\begin{equation}
\displaystyle\sup_{|\nabla u|_2\leq1}\displaystyle\int_{\mathbb{R}^2}\Big(e^{\alpha|u|^2}-1\Big)\hspace{0.05cm}dx\leq C.
\label{T-M}
\end{equation}
The lemma below is a technical result that will be used later on. 
\begin{lemma}\label{lema3.10}  
	Let $(w_n)$ be a sequence in $H^1(\mathbb{R}^2)$ such that $\displaystyle\sup_{n\in\N}||w_n||^2\leq m<4\pi$. For each $\alpha>\alpha_0$ and $q>1$ such that $qm\alpha<4\pi$, there exists a constant $C=C(q,m,\alpha)>0$ such that $b_\alpha(w_n)=(e^{\alpha w_n^2}-1)$ belongs to $L^q(\mathbb{R}^2)$ and 
	$$
	\displaystyle\sup_{n\in\N}|b_\alpha(w_n)|_q<\infty.
	$$
\end{lemma}
\begin{proof} 
As $||w_n||^2\leq m$, for every $n\in\mathbb{N}$. Letting $\alpha>\alpha_0$ and $q>1$ such that $qm\alpha<4\pi$, it is possible to find $q'>q$ close to $q$ such that $q'm\alpha<4\pi$. Thereby, 
	$$
	\begin{array}{ccc}
	|b_\alpha(w_n)|_q^q=\displaystyle\int_{\mathbb{R}^2}|b_\alpha(w_n)|^q\hspace{0.05cm}dx
	&=&\displaystyle\int_{\mathbb{R}^2}(e^{\alpha|w_n|^2}-1)^q\hspace{0.05cm}dx\ \ \ \ \ \ \ \ \ \ \ \  \  \ \ \ \ \ \\\
	&\leq&C_1\displaystyle\int_{\mathbb{R}^2}(e^{q'\alpha|w_n|^2}-1)\hspace{0.05cm}dx \ \ \ \ \ \ \ \ \ \ \ \ \ \\
	&=& C_2\displaystyle\int_{\mathbb{R}^2}(e^{q'\alpha||w_n||_\lambda\big(\frac{|w_n|}{||w_n||_\lambda}\big)^2}-1)\hspace{0.05cm}dx\\
	&\leq& C_3\displaystyle\int_{\mathbb{R}^2}(e^{q'\alpha m\big(\frac{|w_n|}{||w_n||_\lambda}\big)^2}-1)\hspace{0.05cm}dx. \ \ \
	\end{array}
	$$
By (\ref{T-M}), there exists $C=C(q,m,\alpha)>0$ such that 
	$$
	|b_\alpha(w_n)|_q^q=\displaystyle\int_{\mathbb{R}^2}|b_\alpha(w_n)|^q\hspace{0.05cm}dx\leq C.
	$$
\end{proof}

Next, we will show some technical lemmas that are crucial in our approach. 

\begin{lemma} \noindent $(a)$ \, {There are constants $r_\nu,\rho_\nu >0$ with $r_\nu \to 0$ and $\rho_\nu \to 0$ as $\nu \to +\infty$}, independent of $\lambda$ and $\epsilon$, such that 
	$$
{I_{\lambda,\epsilon}(u) \geq \rho_\nu \quad \mbox{for} \quad \|u\|_{\lambda}=r_\nu;}
	$$
	\noindent $(b)$ \, There is $e \in H^1(\mathbb{R}^2)$ with {$\|e\|_{\lambda} >r_\nu$} and $I_{\lambda, \epsilon}(e)<0$.
	\label{lema8}
\end{lemma}
\begin{proof} Fixing $l> 1$ and $\alpha>\alpha_0$, by $(\tilde{g}_3)$ there exists $C_\nu>0$ with $C_\nu \to \infty$ as $ \nu \to \infty$ such that
	\begin{equation}
	|\tilde{g}(x,t)|\leq \frac{1}{2}|t|+C_\nu|t|^l\Big(e^{\alpha|t|^2}-1\Big), \ \forall (x,t)\in\mathbb{R}^2\times \R.
	\label{eq2.40}
	\end{equation}
Thereby, by H\"older inequality and Sobolev embedding  
	\begin{equation}
	I_{\lambda,\epsilon}(u)\geq \frac{1}{4}||u||^2_\lambda-C_\nu|u|_{(l+1)s_1}^{l+1}\Big|e^{\alpha|u|^2}-1\Big|_{s_2},
	\label{eq2.44}
	\end{equation}
	with 
	$$
	\frac{1}{s_1} + \frac{1}{s_2} = 1.
	$$
	Now, let us fix $r_0>0$ such that $\alpha r_0^2<4\pi$, $s_2>1$ close to $1$ satisfying $s_2\alpha r_0^2<4\pi$, and $s_2'>s_2$ such that $s_2'\alpha r_0^2<4\pi$. Then, for $u\in E_\lambda$ with $||u||_\lambda<r_0$, 
	$$
	\begin{array}{ccc}
	\displaystyle\int_{\mathbb{R}^2}\big(e^{\alpha|u|^2}-1\big)^{s_2}\hspace{0.05cm}dx&\leq& C\displaystyle\int_{\mathbb{R}^2}(e^{s_2'\alpha|u|^2}-1)\hspace{0.05cm}dx \ \ \ \ \ \ \ \ \\
	&=&C\displaystyle\int_{\mathbb{R}^2}(e^{s_2'\alpha||u||_\lambda^2(\frac{|u|}{||u||_\lambda})^2}-1)\hspace{0.05cm}dx \\
	&\leq&C\displaystyle\int_{\mathbb{R}^2}(e^{s_2'\alpha r_0^2(\frac{|u|}{||u||_\lambda})^2}-1)\hspace{0.05cm}dx. \ \ \ \\
	\end{array}
	$$
Then by (\ref{T-M}), there is $C'>0$ such that
	\begin{equation}
	\Big|e^{\alpha|u|^2}-1\Big|_{s_2}\leq C', \ \ \text{for} \ \  ||u||_\lambda\leq r_0.
	\label{eq2.45}
	\end{equation}
Now, (\ref{eq2.44}) together with (\ref{eq2.45})  and Sobolev embeddings leads to  
	$$
	I_{\lambda,\epsilon}(u)\geq \frac{1}{4}||u||^2_\lambda-C_\nu||u||^{l+1}_\lambda.
	$$
Now, arguing as in the proof of Lemma \ref{mpgcon}, there are $r_\nu, \rho_\nu>0$ with $r_\nu, \rho_\nu \to 0$ as $\nu \to +\infty$ such that 
	$$
	I_{\lambda,\epsilon}(u)\geq \rho_\nu, \ \ \ \ \text{for} \ \ ||u||_\lambda=r_\nu,
	$$
which shows $(a)$.	The item $(b)$ follows as in the proof of Lemma \ref{mpgcon}.
\end{proof}

By Lemma \ref{lema8} and the mountain pass theorem due to Willem \cite{Willem}, there exists a
$(PS)c_{\lambda,\epsilon}$ sequence $(u_n)$ for $I_{\lambda,\epsilon}$ such that 
$$
I_{\lambda,\epsilon}(u_n)\rightarrow c_{\lambda,\epsilon} \;\; \text{and} \;\; I_{\lambda,\epsilon}'(u_n)\rightarrow 0,
$$
where
$$c_{\lambda,\epsilon}=\inf_{\gamma\in\Gamma}\max_{t\in[0,1]}I_{\lambda,\epsilon}(\gamma(t))$$
with 
$$
\Gamma=\{\gamma\in C([0,1],E_\lambda); \ \gamma(0)=\varphi^+ \ \text{and} \ \gamma(1)=\widetilde{w}\}.
$$

In the same way as in the proof of Lemma \ref{boundps}, we can prove the lemma below

\begin{lemma} \label{lema3.7}
If $(w_n)$ is a $(PS)_{c_{\lambda,\epsilon}}$ sequence for the functional $I_{\lambda,\epsilon}$, then $(w_n)$ is bounded 
\end{lemma}
	
Reasoning  as in Lemma \ref{boundps}, we can also assume that $(w_n)$ is formed by non-negative functions and  
\begin{equation} \label{Z1001}
\displaystyle\limsup_{n\rightarrow\infty} ||w_n||^2_\lambda\leq \frac{4\theta}{\theta -2}c_{\lambda,\epsilon}.
\end{equation}

\begin{lemma}\label{lema3.8} 
If $(u_n)$ is a $(PS)_{c_{\lambda, \epsilon}}$ sequence for the functional $I_{\lambda,\epsilon}$, then, given $\delta>0$ there is  $R>0$ such that
$$
\displaystyle\limsup_{n\rightarrow\infty}  \Bigg(\int_{B^c(0, R)}(|\nabla u_n|^2+|u_n|^2)\hspace{0.05cm}dx\Bigg)<\delta
$$
\end{lemma}
\begin{proof} See proof of Lemma \ref{vanish}.
\end{proof}	
\begin{lemma} \label{lema3.9}
Let $c_{\lambda,\epsilon}$ be the mountain pass level for the functional $I_{\lambda,\epsilon}$. Then, there exists a constant $A>0$, independent of $\nu$, such that 
$$
0<c_{\lambda,\epsilon}<A/\nu^{\frac{2}{p-1}}.
$$
\end{lemma}
\begin{proof} See \cite[Lemma 4.6]{CALBCT}.
\end{proof}	

\begin{proposition}\label{prop3.3}
For $\lambda\geq 1$, the functional  $I_{\lambda,\epsilon}$ satisfies the $(PS)_{d}$ sequence for $d \in (0, \frac{\theta-2}{\theta}\frac{\pi}{\alpha_0})$. 
\end{proposition}
\begin{proof}
Let $(w_n)$ be a $(PS)_{d}$ sequence for the functional $I_{\lambda,\epsilon}$ with $d \in (0,\frac{\theta-2}{\theta}\frac{\pi}{\alpha_0})$. By Lemma \ref{lema3.7}, $(w_n)$ is bounded, and so, there is $w\in E_\lambda$ such that, up to a subsequence, 
$$
\begin{aligned}
&w_n\rightharpoonup w\ \text{in} \ E_\lambda,\\
&w_n(x)\rightarrow w \ \text{a.e.} \ \text{in} \ \mathbb{R}^2\;\;\mbox{and}\\
&w_n\rightarrow w \ \text{in} \ L_{Loc}^s(\mathbb{R}^2), \forall s\geq2.
\end{aligned}
$$
Note that 
$$
||w_n-w||_\lambda^2=\langle w_n-w,w_n-w\rangle_\lambda=||w_n||^2_\lambda-\langle w_n,w\rangle_\lambda+o_n(1),
$$
$$
||w_n||^2_\lambda=I_{\lambda,\epsilon}'(w_n)w_n+\frac{1}{\epsilon} \int_\Omega(\varphi-w_n)^+w_n\,dx+\int_{\mathbb{R}^N}g(x,w_n) w_n\,dx
$$
and
$$
-\langle w_n,w\rangle_\lambda=-I_{\lambda,\epsilon}'(w_n)w-\frac{1}{\epsilon} \int_\Omega(\varphi-w_n)^+w\,dx-\int_{\mathbb{R}^N}g(x,w_n) w\,dx.
$$
As in the proof of \cite[Claim 3.5]{CALBCT},  
	$$
	\int_\Omega(\varphi-w_n)^+w_n\,dx-\int_\Omega(\varphi-w_n)^+w\,dx=o_n(1).
	$$
So, 
	$$
	||w_n-w||_\lambda^2=\int_{\mathbb{R}^2}g(x,w_n) w_n\,dx-\int_{\mathbb{R}^2}g(x,w_n) w\,dx+o_n(1).
	$$
\begin{claim}\label{afir3.4} 
The following limits are true:
\begin{itemize}
\item [$(i)$] $ \displaystyle\lim_{n\rightarrow\infty}\displaystyle\int_{\mathbb{R}^2}\widetilde{g}(x,w_n) w_n\,dx=\int_{\mathbb{R}^N}\widetilde{g}(x,w) w\,dx;$	
\item[$(ii)$]$ \displaystyle\lim_{n\rightarrow\infty}\displaystyle\int_{\mathbb{R}^2}\widetilde{g}(x,w_n) v\,dx=\int_{\mathbb{R}^N}\widetilde{g}(x,w) v\,dx, \quad \forall v\in E_\lambda.$
			\end{itemize}
\end{claim}
We start by showing $(i)$. Note that   
$$
\Bigg|\displaystyle\int_{\mathbb{R}^2}\widetilde{g}(x,w_n) w_n\,dx-\int_{\mathbb{R}^2}\widetilde{g}(x,w) w\,dx\Bigg|\leq\displaystyle\int_{\mathbb{R}^2}|g(x,w_n) w_n-g(x,w) w|\,dx.
$$
Given $\delta>0$, consider $R>0$ as in Lemma \ref{lema3.8} and  
$$
I_{n,1}=\int_{B(0, R)}|\widetilde{g}(x,w_n)w_n-\widetilde{g}(x,w)w|\,dx\quad\text{and}\quad I_{n,2}=\int_{B^c(0, R)}|\widetilde{g}(x,w_n)w_n-\widetilde{g}(x,w)w|\,dx.
$$
By $(\widetilde{g}_5)$,  
$$
|\widetilde{g}(x,w_n)|\leq |w_n|^2+C|w_n|b_\alpha(w_n), \quad \forall x\in\mathbb{R}^2 \ \ \text{and} \ \ n\in\mathbb{N}.
$$
Our goal is to show that $\displaystyle \lim_{n \to +\infty}I_{n,1}=0$. In the sequel, let us define the following functions 
	$$
	\phi_n:=|w_n|^2+C|w_n|b_\alpha(w_n)\ \ \ \text{and} \ \ \ \phi:=|w|^2+C|w|b_\alpha(w)
	$$
Since 
	$$
	w_n(x)\rightarrow w(x) \ \text{a.e.} \ \text{in} \ \mathbb{R}^2,
	$$ 
then 
	$$
	\widetilde{g}(x,w_n)\rightarrow\widetilde{g}(x,w) \ \text{a.e.} \ \text{in} \ B(0, R)
	$$
	and
	$$
	\phi_n(x)\rightarrow \phi(x) \ \text{a.e.} \ \text{in} \ B(0, R).
	$$
In order to conclude our result, we need to show that 
	$$
	\phi_n\rightarrow \phi  \ \text{in} \ L^1(B(0, R)).
	$$
As $ \displaystyle\limsup_{n\rightarrow\infty}||w_n||_{\lambda}^2< 4\pi/\alpha_0,$ there are $0<m<4\pi/\alpha_0$ and $n_0 \in \mathbb{N}$ such that 
	$$
	||w_n||_\lambda^2\leq m, \quad \forall n \geq n_0.
	$$
Since $m\alpha_0<4\pi$, then, for $\alpha>\alpha_0$ and $q>1$ close enough to $\alpha_0$ and $1$ respectively, we must have $qm\alpha<4\pi$. By Lemma \ref{lema3.10}, there is $C>0$ such that  $b_\alpha(w_n)=(e^{\alpha w_n^2}-1)$ belongs to $L^q(\mathbb{R}^2)$ and
	$$
	|b_\alpha(w_n)|_q\leq C, \quad \forall n\in\mathbb{N}.
	$$
Hence, the sequence $(b_\alpha(w_n))$ is bounded in $L^q(B(0, R))$, and by \cite[Lemma 4.8]{Kavian},  
	$$
	b_\alpha(w_n)\rightharpoonup b_\alpha(w)  \ \text{in} \ L^q(B(0, R)).
	$$
As
	$$
	w_n\rightarrow u  \ \text{in} \ L^{q'}(B(0, R)),\quad \text{with} \ \ 1/q+1/q'=1,
	$$
it follows that
	$$
	w_nb_\alpha(w_n)\rightarrow wb_\alpha(w)  \ \text{in} \ L^1(B(0, R)).
	$$
Consequently, 
	$$
	\phi_n\rightarrow \phi  \ \text{in} \ L^1(B(0, R)).
	$$
Therefore, by the Lebesgue dominated convergence theorem  
\begin{equation}	\label{eq3.18}
	\displaystyle\lim_{n\rightarrow\infty}I_{n,1}=0.
\end{equation}
On the other hand, since $\widetilde{\Omega}\subset B(0, R)$, 
$$
|\widetilde{g}(x,t)t|\leq \frac{1}{k}|t|^2,\quad \forall x\in\mathbb{R}^N\setminus B(0,R), \forall t\in\mathbb{R}.
$$
Thus, 
$$
\begin{array}{ll}
	|I_{n,2}| \leq\displaystyle\frac{1}{k}\int_{B^c(0, R)}|w_n|^2\,dx+\frac{1}{k}\int_{B^c_R(0)}|w|^2\,dx\\
	\hspace{0.5cm}\leq\displaystyle\frac{1}{k}\int_{B^c(0, R)}(|\nabla w_n|^2+|w_n|^2)\,dx+\frac{1}{k}\int_{B^c(0, R)}|w|^2\,dx.
	\end{array}
$$
Moreover, as $w\in L^2(\mathbb{R}^N)$, then for $R$ large enough we have 
	$$
	\int_{B^c(0, R)}|w|^2\,dx<\displaystyle\frac{\delta}{2k},
	$$
and by Lemma \ref{lema3.8}, 
	$$
	\displaystyle\limsup_{n\rightarrow\infty} \Bigg(\int_{B^c(0, R)}(|\nabla w_n|^2+|w_n|^2)\hspace{0.05cm}dx\Bigg)<\displaystyle\frac{\delta}{2k}.
	$$
Consequently 
	$$
	\displaystyle \limsup_{n\rightarrow\infty} |I_{n,2}| \leq \delta, \ \ \forall\delta>0,
	$$
which implies 
	\begin{equation}
	\displaystyle\lim_{n\rightarrow\infty}I_{n,2}=0.
	\label{eq3.19}
	\end{equation}
By (\ref{eq3.18}) and (\ref{eq3.19}),
	$$
	\displaystyle\lim_{n\rightarrow\infty}\displaystyle\int_{\mathbb{R}^2}\widetilde{g}(x,w_n) w_n\,dx=\int_{\mathbb{R}^N}\widetilde{g}(x,w) w\,dx.
	$$
A similar argument works to show that $(ii)$ holds.
\end{proof}

Now we are going to show that the penalized problem (\ref{penalizado2}) has a nontrivial weak solution. Let  $c_{\lambda,\epsilon}$ the mountain pass level. By Lemma \ref{lema3.9} there is $\nu_*>0$ such that
$$
0<c_{\lambda,\epsilon}<\frac{\theta-2}{\theta}\frac{\pi}{\alpha_0}, \ \ \text{for} \ \ \nu\geq \nu_*,  \ \ \forall \lambda>0 \ \ \text{and} \ \ \forall \epsilon>0.
$$
By Proposition \ref{prop3.3}, the functional $I_{\lambda,\epsilon}$ satisfies the $(PS)_{c_{\lambda,\epsilon}}$ condition. Therefore, by the Mountain Pass Theorem due to Ambrosetti and Rabinowitz \cite{AR}, there exists $u\in E_\lambda$ such that
$$
I_{\lambda,\epsilon}(u)=c_{\lambda,\epsilon}>0 \ \ \text{and} \ \ I'_{\lambda,\epsilon}(u)=0
$$

\subsection{Existence of solution to the modified variational inequality}

By the previous analysis, we find a nontrivial weak solution $u_\epsilon\in E_\lambda$ to the problem (\ref{penalizado2}), that is, $u_\epsilon$ satisfies 
$$
\displaystyle\int_{\mathbb{R}^2}[\nabla u_\epsilon \nabla v+(1+\lambda V(x))u_\epsilon v]\hspace{0.05cm}dx +\displaystyle\frac{1}{\epsilon} \langle P(u_\epsilon),v\rangle=
\int_{\mathbb{R}^2} \widetilde{g}(x,u_\epsilon)v\hspace{0.05cm}dx, \ \  \forall v \in E_\lambda.
$$
Now, we consider the following notations 
$$
\epsilon=1/n, \ \ \ u_n=u_{1/n} \ \ \ I_{\lambda,\epsilon}=I_n, \ \ \ I_n(u_n)=c_n,
$$ 
where $c_n=c_{\lambda,\epsilon}$. Hence, for each $n\in\mathbb{N}$, there is $u_n\in E_\lambda$ such that 
\begin{equation}
\displaystyle\int_{\mathbb{R}^2}[\nabla u_n \nabla v+(1+\lambda V(x))u_nv]\hspace{0.05cm}dx +n \langle P(u_n),v\rangle=
\int_{\mathbb{R}^2}\widetilde{g}(x,u_n)v\hspace{0.05cm}dx, \ \ \forall v \in E_\lambda.
\label{eq3.20.2}
\end{equation}
Again, as before we can suppose that $(u_n)$ satisfies 
$$
\displaystyle\limsup_{n\to\infty}||u_n||^{2}_\lambda<\frac{4\pi}{\alpha_0},
$$
which implies the boundedness of $(u_n)$ in $H^1(\mathbb{R}^2)$. So, there is $u\in H^1(\mathbb{R}^2)$ such that, up to a subsequence we have 
$$
\begin{aligned}
&u_n\rightharpoonup u \ \ \text{in}  \ \ H^1(\mathbb{R}^2),\\
&u_n\rightarrow u \ \text{in} \ L_{Loc}^s(\mathbb{R}^2), \  s\geq2\;\;\mbox{and}\\
&u_n(x)\rightarrow u(x) \ \text{a.e.} \ \text{in} \ \mathbb{R}^2.
\end{aligned}
$$
Reasoning as in \cite[Lemma 3.11]{CALBCT}, we also have $P(u)=0$, that is, $u\in\mathbb{K}$. 

The next lemma is crucial to show that $u$ is a solution of (\ref{modificado2}).

\begin{lemma}\label{lema3.11} 
the following convergence 
	$$
	u_n\rightarrow u \ \ \ \text{in} \ \ \ E_\lambda.
	$$
holds 
\end{lemma}
\begin{proof} First of all, we recall that 
	\begin{equation*}
	||u_n-u||_\lambda^2=\langle u_n-u,u_n-u\rangle_\lambda=||u_n||^2_\lambda-\langle u_n,u\rangle_\lambda+o_n(1),
	\label{2seq1}
	\end{equation*}
	\begin{equation*}	
	||u_n-u||_\lambda^2=n\langle P(u_n),u-u_n\rangle+\int_{\mathbb{R}^2}\widetilde{g}(x,u_n) u_n\,dx-\int_{\mathbb{R}^2}\widetilde{g}(x,u_n) u\,dx+o_n(1),
	\label{2seq2}
	\end{equation*}
$$
||u_n||^2_\lambda =I_{n}'(u_n)u_n-n\langle P(u_n),u_n\rangle+\int_{\mathbb{R}^2}\widetilde{g}(x,u_n) u_n\,dx, 
$$
$$
-\langle u_n,u\rangle_\lambda =-I_{n}'(u_n)u+n\langle P(u_n),u\rangle-\int_{\mathbb{R}^2}\widetilde{g}(x,u_n) u\,dx
$$
and
$$
\langle P(u_n),u-u_n\rangle =\langle P(u_n)-P(u),u-u_n\rangle\leq0.
$$
All of this information leads to
	\begin{equation}	
	||u_n-u||_\lambda^2\leq \int_{\mathbb{R}^2}\widetilde{g}(x,u_n) u_n\,dx-\int_{\mathbb{R}^2}\widetilde{g}(x,u_n) u\,dx+o_n(1).	
	\label{2seq3}
	\end{equation}
Again, we need the following limits 
	\begin{equation} 
	\displaystyle\lim_{n\rightarrow\infty}\displaystyle\int_{\mathbb{R}^2}\widetilde{g}(x,u_n) u_n\,dx=\int_{\mathbb{R}^2}\widetilde{g}(x,u) u\,dx
	\label{seq6}
	\end{equation}
	and
	\begin{equation}
	\displaystyle\lim_{n\rightarrow\infty}\displaystyle\int_{\mathbb{R}^2}\widetilde{g}(x,u_n) v\,dx=\int_{\mathbb{R}^2}\widetilde{g}(x,u) v\,dx, \quad \forall v\in H^1(\mathbb{R}^2)
	\label{seq7}
	\end{equation}
to finish our proof. With a little bit modification in the proof of Proposition \ref{prop3.3}, it is possible to check that the limits mentioned above are true. More precisely, we note that the sequence $(u_n)$ satisfies the same result as in Lemma \ref{lema3.8}, that is, given $\delta>0$ there is $R>0$ such that
	$$
	\displaystyle\limsup_{n\rightarrow\infty} \Bigg(\int_{B^c(0, R)}(|\nabla u_n|^2+|u_n|^2)\hspace{0.05cm}dx\Bigg)<\delta.
	$$
Furthermore, by (\ref{T-M}), it is possible to show that $b_\alpha(u_n)=(e^{\alpha u_n^2}-1)$ belongs to $L^q(\mathbb{R}^2)$ and
	$$
	\displaystyle\sup_{n\in\N}|b_\alpha(u_n)|_q<\infty
	$$
Now, (\ref{seq6}) - (\ref{seq7}) follow by repeating the same computations explored in the proof of Proposition \ref{prop3.3}.
\end{proof}

Recalling the properties of the penalized operator, we know that   
$$
\langle P(u_n), w-u_n\rangle=\langle P(u_n)-P(w), w-u_n\rangle\leq 0, \quad \forall w\in\mathbb{K}.
$$
Putting $v=w-u_n$ in (\ref{eq3.20}), it follows that 
\begin{equation}
\displaystyle\int_{\mathbb{R}^2}\nabla u_n \nabla (w-u_n)\hspace{0.05cm}dx+\displaystyle\int_{\mathbb{R}^2}(1+\lambda V(x))u_n(w-u_n)\hspace{0.05cm}dx\geq
\int_{\mathbb{R}^2}\widetilde{g}(x,u_n)(w-u_n)\hspace{0.05cm}dx.
\label{eq3.21.2}
\end{equation}
Taking the limit $n\rightarrow+\infty$ in (\ref{eq3.21.2}) and applying Lemma \ref{lema3.11},  we get
\begin{equation}
\displaystyle\int_{\mathbb{R}^2}\nabla u \nabla (w-u)\hspace{0.05cm}dx+\displaystyle\int_{\mathbb{R}^2}(1+\lambda V(x))u(w-u)\hspace{0.05cm}dx\geq
\int_{\mathbb{R}^2}\widetilde{g}(x,u)(w-u)\hspace{0.05cm}dx, \quad \forall w\in\mathbb{K}.
\label{eq3.22}
\end{equation}
This shows that $u$ is a non-negative solution of (\ref{modificado2}).

\subsection{Proof of Theorem \ref{main1}}

In this section, we are going to show that the solutions that were found for the problem (\ref{modificado2}) are in fact solutions for the original problem (\ref{01}) when $\lambda$ is large enough. 
 
Let $u_n\in E_{\lambda_n}$ and $\lambda_n\rightarrow+\infty$ satisfying  
\begin{equation}
\displaystyle\int_{\mathbb{R}^2}\nabla u_n \nabla (v-u_n)\hspace{0.05cm}dx+\displaystyle\int_{\mathbb{R}^2}(1+\lambda_n V(x))u_n(v-u_n)\hspace{0.05cm}dx\geq
\int_{\mathbb{R}^2}\widetilde{g}(x,u_n)(v-u_n)\hspace{0.05cm}dx, \quad \forall v\in\mathbb{K}.
\label{eq3.23}
\end{equation}

The following result below is a key point in our arguments. 
\begin{proposition}\label{prop3.4}
Let $(u_n)$ be a solution sequence of (\ref{eq3.23}). Then, there is a subsequence still denoted by $(u_n)$ and $u \in H^1(\mathbb{R}^2)$ such that
	$$
	u_n \rightharpoonup u\ \mbox{in}\  H^1(\mathbb{R}^2).
	$$
Moreover,
	\begin{itemize}
		\item[$(i)$] $u\equiv 0$ in $\mathbb{R}^2 \setminus \Omega$. 
		\item[$(ii)$] $|| u_n-u||^{2}_{\lambda_{n}} \rightarrow 0.$ 
		\item[$(iii)$] As $\lambda_n\rightarrow \infty$, the following limits
		\begin{align*}
		& u_n \to u \quad \mbox{in} \quad H^{1}(\mathbb{R}^2), &\\
		&\lambda_n \int_{\mathbb{R}^2} V(x)\left|u_n\right|^2dx \rightarrow 0,&\\
		&||u_n||^2_{\lambda_n } \rightarrow \int_{\Omega}(|\nabla u|^2+|u|^2)\,dx=||u||^{2}_{H^{1}(\Omega)}.&
		\end{align*}
		hold.
		\item[$(iv)$]  $u$ be a solution of the variational inequality 
		\begin{equation}
		\displaystyle\int_{\Omega}\nabla u \nabla (v-u)\hspace{0.05cm}dx+\displaystyle\int_{\Omega}u(v-u)\hspace{0.05cm}dx\geq
		\int_{\Omega}f(u)(v-u)\hspace{0.05cm}dx,
		\label{obstaculolimite2}
		\end{equation}
		for all $v \in \tilde{\mathbb{K}}$, where
		$$
		\tilde{\mathbb{K}} = \lbrace v \in H^1_0(\Omega); \ v \geq \varphi \ \text{a.e. in}\ \Omega \rbrace
		$$
		\end{itemize}
\end{proposition}

The main difficulty in the proof of Proposition \ref{prop3.4}, appears in item $(ii)$, because we need that the sequence $(u_n)$ satisfies the same results as in Lemmas \ref{lema3.10} and \ref{lema3.8}. However, analyzing all the results done at this moment, we can guarantee that these results are also true, and so, $(ii)$ holds for $N=2$. To show $(i)$ we follow the same steps given in the proof of \cite[Proposition 3.12]{CALB} and the items $(iii)$ and $(iv)$ are an immediate consequence of item $(i)$ as we can see in the proof of Proposition \ref{propo}. By these reasons, we have omitted their proofs.

Now, we are going to show that there is $\lambda^*>0$ such that, if $\lambda\geq\lambda^*$, the solution $u_\lambda$ of (\ref{modificado2}) satisfies 
\begin{equation} u_\lambda(x) \leq a, \ \forall x \in \mathbb{R}^N\setminus \widetilde{\Omega}.
\label{eq3.24}
\end{equation}
This estimate permits to conclude that $u_\lambda$ is a solution of the variational inequality (\ref{01}). To do this, we will again use the Moser Iteration method \cite{JM} based in the works of Gongbao \cite{LG}, Alves and Souto \cite{AS}, and Alves and Pereira \cite{AP}. Since the case $N=2$ is more subtle than the case $N\geq 3$, we will write a little bit how we can obtain the estimate (\ref{eq3.24}).   

As in the proof of Theorem \ref{main1} for the case $N\geq3$, for each $n\in\N$ and $L\geq1$, let us set
$$
u_{n,L}(x)=\left\{\
\begin{array}{ll}
u_n, \ \text{if}\ \ u_n\leq L\vspace{0.5cm}\\
L, \ \text{if}\ \ u_n\geq L
\end{array}
\right.
$$
Let $d=\text{dist}(\overline{\Omega},\partial\widetilde{\Omega})$ and consider $0<R<d$. Given $x_0\in {\widetilde{\Omega}}^c$, we fixed a sequence $(r_j)_{j\in\N}\subset\R$ such that $R<r_j<d$, $\forall\in\N$ and  $r_j\downarrow R$. Also, consider 
$$
v_{n,L}(x)=v_{n,L,j}(x)=\Big(u_n-u_n\eta^2\Big(\displaystyle\frac{u_{n,L}}{L}\Big)^{2(\beta-1)}\Big)(x),
$$
and 
$$ 
U_{n,L}(x)=U_{n,L,j}(x)=u_n\eta u_{n,L}^{\beta-1}(x)
$$
where, $\eta=\eta_j\in C^{\infty}(\R)$ is such that $0\leq\eta_j\leq1$, $|\nabla\eta_j|\leq\displaystyle\frac{1}{r_j}$ and
$$
\eta_j(x)=\left\{
\begin{array}{ll}
1, \ \text{if}\ \ x\in B_{r_{j+1}}(x_0)\vspace{0.5cm}\\
0, \ \text{if}\ \ x\in B^c_{r_{j}}(x_0)
\end{array}
\right.
$$
and the number $\beta>1$ will be fixed.

Without any difficulty, as in the proof of Claim \ref{claim5} we can show that $v_{n,L}\in\mathbb{K}$ for each $n\in\mathbb{K}$. Considering $v_{n,L}$ as a test function in (\ref{eq3.23}),  
%$$
%\displaystyle\frac{1}{L^{2(\beta-1)}}\Big\{\int_{\R^2}\nabla u_n\nabla(u_n\eta^2u_{n,L}^{2(\beta-1)})+\int_{\R^2}(1+\lambda_nV(x))u_n^2\eta^2u_{n,L}^{2(\beta-1)}\Big\}
%\leq\displaystyle\frac{1}{L^{2(\beta-1)}}\int_{\R^2} \widetilde{g}(x,u_n)u_n\eta^2u_{n,L}^{2(\beta-1)}$$
%moreover,
\begin{equation}
\displaystyle\int_{\R^2}\nabla u_n\nabla(u_n\eta^2u_{n,L}^{2(\beta-1)})\leq-\displaystyle\int_{\R^2}(1+\lambda_nV(x))u_n^2\eta^2u_{n,L}^{2(\beta-1)}\vspace{0.3cm}+\displaystyle\int_{\R^2} \widetilde{g}(x,u_n)u_n\eta^2u_{n,L}^{2(\beta-1)}.\hspace{0.7cm}
\end{equation}
Since $\lambda_nV(x)\geq0$ for all $x\in\mathbb{R}$, 
\begin{equation}
\displaystyle\int_{\R^2}\nabla u_n\nabla(u_n\eta^2u_{n,L}^{2(\beta-1)})\leq -\displaystyle\int_{\R^2}u_n^2\eta^2u_{n,L}^{2(\beta-1)}+\displaystyle\int_{\R^2} \widetilde{g}(x,u_n)u_n\eta^2u_{n,L}^{2(\beta-1)}.\hspace{0.7cm}
\label{eq3.25}
\end{equation}
On the other hand, as  
\begin{equation}\label{eq3.26}
\begin{array}{ccc}
\displaystyle\int_{\R^2}\nabla u_n\nabla(u_n\eta^2u_{n,L}^{2(\beta-1)})&=&\displaystyle\int_{\R^2}|\nabla u_n|^2\eta^2u_{n,L}^{2(\beta-1)}+2\displaystyle\int_{\R^2}\nabla u_n\nabla\eta u_n\eta u_{n,L}^{2(\beta-1)}\vspace{0.2cm}\vspace{0.3cm}\\
%-2(\beta-1)\displaystyle\int_{u_n\leq L}|\nabla u_n|^2\eta^2u_{n,L}^{2(\beta-1)}
&&+2(\beta-1)\displaystyle\int_{\R^2}\nabla u_n\nabla u_{n,L} u_n\eta^2u_{n,L}^{2(\beta-1)-1}
\end{array}
\end{equation}
and 
\begin{equation}\label{eq3.27}
2(\beta-1)\displaystyle\int_{\R^N}\nabla u\nabla u_{n,L} u_n\eta^2u_{n,L}^{2(\beta-1)-1}=2(\beta-1)\displaystyle\int_{u_n\leq L}|\nabla u_n|^2\eta^2u_{n,L}^{2(\beta-1)}\geq0,
\end{equation}
(\ref{eq3.25}) combined with (\ref{eq3.26}) and (\ref{eq3.27}) gives  
\begin{equation}\label{eq3.28}
\displaystyle\int_{\R^2}|\nabla u_n|^2\eta^2u_{n,L}^{2(\beta-1)}+2\displaystyle\int_{\R^2}\nabla u_n\nabla\eta u_n\eta u_{n,L}^{2(\beta-1)}\leq-\displaystyle\int_{\R^2}u_n^2\eta^2u_{n,L}^{2(\beta-1)}+\displaystyle\int_{\R^2} \widetilde{g}(x,u_n)u_n\eta^2u_{n,L}^{2(\beta-1)}.\hspace{0.7cm}
\end{equation}
By $(\widetilde{g}_3)$, fixed $\alpha\geq\alpha_0$ there is $C>0$ such that
$$
|\widetilde{g}(x,t)|\leq |t|+C|t|b_\alpha(t), \quad \forall (x,t)\in\mathbb{R}^2\times\mathbb{R}.
$$
Hence, (\ref{eq3.28}) can be rewritten as follows
\begin{equation}
\displaystyle\int_{\R^2}|\nabla u_n|^2\eta^2u_{n,L}^{2(\beta-1)}+2\displaystyle\int_{\R^2}\nabla u_n\nabla\eta u_n\eta u_{n,L}^{2(\beta-1)}
\leq C\displaystyle\int_{\R^2}b_\alpha(u_n) u_n^{2}\eta^2u_{n,L}^{2(\beta-1)}.
%\label{lim.4}
\end{equation}
By the definition of the function $\eta$, 
\begin{equation}\label{eq3.29}
\displaystyle\int_{B(x_0, {r_{j}})}|\nabla u_n|^2\eta^2u_{n,L}^{2(\beta-1)}\,dx+2\displaystyle\int_{B(x_0, {r_{j}})}\nabla u_n\nabla\eta u_n\eta u_{n,L}^{2(\beta-1)}
\leq C\displaystyle\int_{B(x_0, {r_{j}})}b_\alpha(u_n)\eta^2u_{n,L}^{2(\beta-1)}.
\end{equation}
Furthermore, by Sobolev embedding 
\begin{equation}\label{imersaogama}
|U_{n,L}|^2_{L^\gamma(B(x_0, {r_{j}}))}\leq C_\gamma\displaystyle\int_{B(x_0, {r_{j}})}\Big(|\nabla U_{n,L}|^2+|U_{n,L}|^2\Big)\,dx
\end{equation}
for any $\gamma\geq2$. Note that 
\begin{equation}
\begin{aligned}
\displaystyle\int_{B(x_0, {r_{j}})}|\nabla U_{n,L}|^2&\leq \displaystyle\int_{B(x_0, {r_{j}})}|\nabla \eta|^2u_n^2u_{n,L}^{2(\beta-1)}+
3\beta^2\displaystyle\int_{B(x_0, {r_{j}})}|\nabla u_n|^2\eta^2u_{n,L}^{2(\beta-1)}\\
&+6\beta^2\displaystyle\int_{B(x_0, {r_{j}})}\nabla u_n\nabla\eta u_n\eta u_{n,L}^{2(\beta-1)}.\hspace{2.5cm}\\
\end{aligned}
\label{eq3.30-}
\end{equation}
Replacing (\ref{eq3.29}) in (\ref{eq3.30-}),  
$$
\displaystyle\int_{B(x_0, {r_{j}})}|\nabla U_{n,L}|^2\,dx\leq\displaystyle\int_{B(x_0, {r_{j}})}|\nabla \eta|^2u_n^2u_{n,L}^{2(\beta-1)}\,dx
+3\beta^2\displaystyle\int_{B(x_0, {r_{j}})}b_\alpha(u_n)\eta^2u_n^2u_{n,L}^{2(\beta-1)}\,dx.  
$$
By using the properties of function $\eta$ it follows 
$$
\displaystyle\int_{B(x_0, {r_{j}})}|\nabla U_{n,L}|^2\,dx\leq C\displaystyle\int_{B(x_0, {r_{j}})}u_n^2u_{n,L}^{2(\beta-1)}\,dx
+3\beta^2\displaystyle\int_{B(x_0, {r_{j}})}b_\alpha(u_n)\eta^2u_n^2u_{n,L}^{2(\beta-1)}\,dx.  
$$
that is,
\begin{equation}
\displaystyle\int_{B(x_0, {r_{j}})}|\nabla U_{n,L}|^2\,dx\leq C\displaystyle\int_{B(x_0, {r_{j}})}u_n^2u_{n,L}^{2(\beta-1)}\,dx
+3\beta^2\displaystyle\int_{B(x_0, {r_{j}})}b_\alpha(u_n)\eta^2u_n^2u_{n,L}^{2(\beta-1)}\,dx.
\label{eq3.31}
\end{equation}
Replacing (\ref{eq3.31}) in (\ref{imersaogama}) and doing some computations we obtain 
\begin{equation}
|U_{n,L}|^2_{L^\gamma(B(x_0, {r_{j}}))}\leq C\beta^2\Bigg[\displaystyle\int_{B(x_0, {r_{j}})}u_n^2u_{n,L}^{2(\beta-1)}\,dx
+\displaystyle\int_{B(x_0, {r_{j}})}b_\alpha(u_n)\eta^2u_n^2u_{n,L}^{2(\beta-1)}\,dx\Bigg].
\label{eq3.32}
\end{equation}
By Lemma \ref{lema3.10} we know that $b_\alpha(u_n)\in L^q(\mathbb{R}^2)$ for some $q>1$ and $q\approx1$. Moreover, there is $C>0$ such that 
$$
|b_\alpha(u_n)|_q^q\leq C, \ \ \ \forall n\in\mathbb{N}.
$$
This together with H\"older inequality ensures that 
\begin{equation}
\begin{array}{ccc}
|U_{n,L}|_{L^\gamma(B(x_0, {r_j}))}&\leq&C\beta^2\Big[\displaystyle\int_{B(x_0, {r_j})}(u_n^2u_{n,L}^{2(\beta-1)})^{q'}\,dx\Big]^{1/q'}|B(x_0, {r_j})|^{1/q}
\hspace{2.5cm}\\
&+&C\beta^2\Big[\displaystyle\int_{B(x_0, {r_j})}(b_\alpha(u_n))^{q}\,dx\Big]^{1/q}\Big[\displaystyle\int_{B(x_0, {r_j})}(u_n^2u_{n,L}^{2(\beta-1)})^{q'}\,dx\Big]^{1/q'}.
\end{array}
\end{equation}
where $\frac{1}{q}+\frac{1}{q'}=1$. Since $r_{j+1}<r_{j+1}<\cdots<r_{1}$ for all $\forall j\in\mathbb{N}$, 
\begin{equation}
|U_{n,L}|_{L^\gamma(B(x_0, {r_j}))}\leq C\beta^2\Big[\displaystyle\int_{B(x_0, {r_j})}u_n^{2q'}u_{n,L}^{2q'(\beta-1)}\,dx\Big]^{1/q'}
\label{eq.33}
\end{equation}
Consequently,
\begin{equation}
\Big(\displaystyle\int_{B(x_0, {r_{j+1}})}u_{n,L}^{\gamma\beta}\,dx\Big)^{2/\gamma}\leq C\beta^2\Big[\displaystyle\int_{B(x_0, {r_{j}})}u_n^{2q'}u_{n,L}^{2q'(\beta-1)}\,dx\Big]^{1/q'},
\end{equation}
that is,
\begin{equation}
\Big(\displaystyle\int_{B(x_0, {r_{j+1}})}u_{n,L}^{\gamma\beta}\,dx\Big)^{2/\gamma}\leq C\beta^2\Big(\displaystyle\int_{B(x_0, {r_{j}})}u_n^{2q'\beta}\,dx\Big)^{1/q'},
\end{equation}
leading to
\begin{equation}
|u_{n,L}|^{2\beta}_{L^{\gamma\beta}(B(x_0, {r_{j+1}}))}\leq C\beta^2|u_{n}|^{2\beta}_{L^{2q'\beta}(B(x_0, {r_{j}}))}.
\label{eq3.34}
\end{equation}
Applying the Fatou's lemma with respect to the variable $L$ in (\ref{eq3.34}),  we find
\begin{equation}
|u_{n}|_{L^{\gamma\beta}(B(x_0, {r_{j+1}}))}\leq C^{\frac{1}{\beta}}\beta^{\frac{1}{\beta}}|u_{n}|_{L^{2q'\beta}(B(x_0, {r_{j}}))}
\label{eq3.35}
\end{equation}
for all $j,n\in\mathbb{N}$ and $C = C(\gamma)$. After more some calculus, we get
\begin{equation}
|u_{n}|_{L^{\infty}(B_{R}(x_0))}\leq C|u_{n}|_{L^{\gamma}(B_{r_{1}}(x_0))}\leq C|u_{n}|_{L^{\gamma}(\Omega^c)} ,\quad \forall n\in\mathbb{N}.
\label{eq3.45}
\end{equation}
By Proposition \ref{prop3.4}, 
\begin{center}
	$u_n=u_{n,\lambda_n}\rightarrow0$ in $H^1(\Omega^c)$ as $\lambda_n\rightarrow+\infty$
\end{center}
and by Sobolev embedding, given $\kappa>0$ there is $n_0\in\N$ such that
$$
|u_{n}|_{L^{\infty}(B_{R}(x_0))}<\kappa, \ \ \text{for} \ n\geq n_0.
$$
As $x_0\in\Omega'^c$ is arbitrary, next
$$
|u_{n}|_{L^{\infty}(\Omega'^c)}<\kappa.
$$
So, we guarantee the existence of $\lambda^*>0$ such that, the solution $u_\lambda$ satisfies 
\begin{center}
	$u_\lambda(x)\leq a$, $\forall x\in\Omega'^c$ and $\lambda\geq\lambda^*$.
\end{center}
Therefore, for $\lambda\geq\lambda^*$ the function $u$ satisfies  
\begin{equation}
\displaystyle\int_{\mathbb{R}^2}\nabla u \nabla(v-u)\hspace{0.05cm}dx+\displaystyle\int_{\mathbb{R}^2}(1+\lambda V(x))u(v-u)\hspace{0.05cm}dx 
\geq \int_{\mathbb{R}^2}f(u)(v-u)\hspace{0.05cm}dx, \quad \forall v \in \mathbb{K}, 
\label{eq3.46}
\end{equation}	
finishing the proof. 
%%%%%%%%%%%%%%%%%%%%%%%%%%%%%%%%%%%%%%%%%%%%%%%%%

%%%%%%%%%%%%%%%%%%%%%%%%%%%%%%%%%%%%
\end{document}